\documentclass[12pt]{article}
\usepackage{tikz}
\usetikzlibrary{arrows}
\usepackage[framemethod=tikz]{mdframed}
\usepackage{wrapfig}
\usepackage{amsmath,amssymb}
\usepackage{type1cm}
\usepackage{amsthm}
\usepackage{mathrsfs}
\usepackage{mathtools}
\usepackage{enumerate}
\usepackage[all]{xy} 
\usepackage[vcentermath,enableskew]{youngtab}
\usepackage{ytableau}
\usepackage{diagbox}
\AtBeginDocument{%
   \def\MR#1{}
}

\DeclareMathOperator{\GL}{GL}

\DeclareMathOperator{\ch}{ch}

\DeclareMathOperator{\Ad}{Ad}
\DeclareMathOperator{\Adm}{Adm}

\DeclareMathOperator{\Irr}{Irr}
\DeclareMathOperator{\pr}{pr}

\DeclareMathOperator{\pfn}{pfn}

\DeclareMathOperator{\en}{end}

\DeclareMathOperator{\supp}{supp}
\DeclareMathOperator{\de}{def}
\DeclareMathOperator{\LP}{LP}

\newcommand\F{\mathbb{F}}

\newcommand\Fq{\mathbb F_q}
\newcommand\aFq{\overline{\mathbb F}_q}

\newcommand\cB{\mathcal B}
\newcommand\cO{\mathcal O}

\newcommand\cG{\mathcal G}
\newcommand\cF{\mathcal F}
\newcommand\cT{\mathcal T}

\newcommand\Gm{\mathbb G_m}

\newcommand\N{\mathbb N}
\newcommand\Q{\mathbb Q}

\newcommand\A{\mathbb A}

\newcommand\G{\mathbb G}
\newcommand\J{\mathbb J}

\newcommand\Z{\mathbb Z}
\newcommand\tW{\tilde W}

\newcommand\SW{{^S\tilde W}}
\newcommand\SAdm{{^S\mathrm{Adm}}}

\newcommand\tS{\tilde S}
\newcommand\inv{\mathrm{inv}}

\newcommand\ld{\lambda}

\newcommand\unp{\underline p}

\newcommand\vp{\varpi}
\newcommand\Y{X_*(T)}

\newcommand\mn{\lfloor \frac{m}{n} \rfloor}
\newcommand\la{\langle}
\newcommand\ra{\rangle}
\newcommand\pc{\preceq}

\theoremstyle{definition}
\newtheorem{theo}{Theorem}[section]
\newtheorem{prop}[theo]{Proposition}

\newtheorem{lemm}[theo]{Lemma}
\newtheorem{coro}[theo]{Corollary}

\newtheorem{rema}[theo]{Remark}

\newtheorem{thm}{Theorem}[section]

\begin{document}
\title{Basic Loci of Positive Coxeter Type for $\GL_n$}
\author{Ryosuke Shimada}
\date{}
\maketitle

\begin{abstract}
Motivated by the problem of giving an explicit description of the basic locus in the reduction of Shimura varieties, G\"{o}rtz, He and Nie \cite{GHN22} studied the cases where the basic affine Deligne-Lusztig variety, which serves as its group-theoretic model, is a union of classical Deligne-Lusztig varieties associated to Coxeter elements.
In this paper, we study a natural generalization of this stratification in the case of $\GL_n$.
\end{abstract}

\section{Introduction}
\label{introduction}
The affine Deligne-Lusztig variety was introduced by Rapoport in \cite{Rapoport05}, which plays an important role in understanding geometric and arithmetic properties of Shimura varieties.
The uniformization theorem by Rapoport and Zink \cite{RZ96} allows us to describe the Newton strata of Shimura varieties in terms of Rapoport-Zink spaces, whose underlying spaces are special cases of affine Deligne-Lusztig varieties.

Let $F$ be a non-archimedean local field with finite residue field $\F_q$ of prime characteristic $p$, and let $L$ be the completion of the maximal unramified extension of $F$.
Let $\sigma$ denote the Frobenius automorphism of $L/F$.
Further, we write $\cO$ (resp.\ $\cO_F$) for the valuation ring of $L$ (resp.\ $F$).
Finally, we denote by $\vp$ a uniformizer of $F$ (and $L$) and by $v_L$ the valuation of $L$ such that $v_L(\vp)=1$.

Let $G$ be an unramified connected reductive group over $\cO_F$.
Let $B\subset G$ be a Borel subgroup and $T\subset B$ a maximal torus in $B$, both defined over $\cO_F$.
For $\mu,\mu'\in X_*(T)$ (resp.\ $X_*(T)_{\Q}$), we write $\mu'\pc \mu$ if $\mu-\mu'$ is a non-negative integral (resp.\ rational) linear combination of positive coroots.
For a cocharacter $\mu\in X_*(T)$, let $\vp^{\mu}$ be the image of $\vp\in \mathbb G_m(F)$ under the homomorphism $\mu\colon\mathbb G_m\rightarrow T$.

Set $K=G(\cO)$.
We fix a dominant cocharacter $\mu\in X_*(T)_+$ and $b\in G(L)$.
Then the affine Deligne-Lusztig variety $X_{\mu}(b)$ is the locally closed reduced $\aFq$-subscheme of the affine Grassmannian $\cG r=G(L)/K$ defined as
$$X_{\mu}(b)=\{xK\in \cG r\mid x^{-1}b\sigma(x)\in K\vp^{\mu}K\}.$$
The closed affine Deligne-Lusztig variety is the closed reduced $\aFq$-subscheme of $\cG r$ defined as
$$X_{\pc\mu}(b)=\bigcup_{\mu'\pc \mu}X_{\mu'}(b).$$
Both $X_{\mu}(b)$ and $X_{\pc\mu}(b)$ are locally of finite type in the equal characteristic case and locally perfectly of finite type in the mixed characteristic case (cf.\ \cite[Corollary 6.5]{HV11}, \cite[Lemma 1.1]{HV18}).
Also, the affine Deligne-Lusztig varieties $X_{\mu}(b)$ and $X_{\pc\mu}(b)$ carry a natural action (by left multiplication) by the $\sigma$-centralizer of $b$
$$\J=\J_b=\{g\in G(L)\mid g^{-1}b\sigma(g)=b\}.$$
(Since $b$ is usually fixed in the discussion, we mostly omit it from the notation.)

The geometric properties of affine Deligne-Lusztig varieties have been studied by many people.
For example, the non-emptiness criterion and the dimension formula are already known for the affine Deligne-Lusztig varieties in the affine Grassmannian (see \cite{Gashi10}, \cite{Viehmann06} and \cite{Hamacher15}).
Let $B(G)$ denote the set of $\sigma$-conjugacy classes of $G(L)$. 
Thanks to Kottwitz \cite{Kottwitz85}, a $\sigma$-conjugacy class $[b]\in B(G)$ is uniquely determined by two invariants: the Kottwitz point $\kappa(b)\in \pi_1(G)/((1-\sigma)\pi_1(G))$ and the Newton point $\nu_b\in X_*(T)_{\Q,+}$.
Set $B(G,\mu)=\{[b]\in B(G)\mid \kappa(b)=\kappa(\vp^\mu), \nu_b\pc \mu^{\diamond}\}$, where $\mu^{\diamond}$ denotes the $\sigma$-average of $\mu$.
Then $X_\mu(b)\neq \emptyset$ if and only if $[b]\in B(G,\mu)$.
If this is the case, then we have
$$\dim X_\mu(b)=\la\rho, \mu-\nu_b\ra-\frac{1}{2}\de(b),$$
where $\rho$ is the half sum of positive roots and $\de(b)$ is the defect of $b$.
It is also known that (closed) affine Deligne-Lusztig varieties are equidimensional (cf.\ \cite{HV12}, \cite{Takaya22}).
Note that the dimension formula and the equidimensionality imply that $X_{\pc\mu}(b)$ is actually the closure of $X_{\mu}(b)$ in $\cG r$ (cf.\ \cite[Corollary 1.3]{HV11}).

Via the relationship to Shimura varieties, or more directly to Rapoport-Zink spaces, the results on the geometry of affine Deligne-Lusztig varieties have numerous applications to number theory (e.g., the Kudla-Rapoport program \cite{KR11}, Zhang's Arithmetic Fundamental Lemma \cite{Zhang12}, $\ldots$).
Many of these applications make use of the nice cases where the basic locus of $X_{\pc\mu}(b)$ (or the special fiber of the corresponding Shimura variety) admits a simple description.
The fully Hodge-Newton decomposable case, introduced by G\"{o}rtz, He and Nie in \cite{GHN19}, is one of such cases.
They proved that if $(G, \mu)$  is fully Hodge-Newton decomposable, then $X_{\pc \mu}(\tau_\mu)$ is naturally a union of (classical) Deligne-Lusztig varieties, where $\tau_\mu$ is (a representative in $G(L)$ of) a length $0$ element in the Iwahori-Weyl group $\tW$ of $T$ whose $\sigma$-conjugacy class in $G(L)$ is the unique basic element in $B(G,\mu)$.
This stratification is the so-called weak Bruhat-Tits stratification, a stratification indexed in terms of the Bruhat-Tits building of $\J$ (which exists only in the fully Hodge-Newton decomposable case).
The case of Coxeter type is a special case of this case such that each Deligne-Lusztig variety appearing in this stratification is of Coxeter type (cf.\ \cite[\S 2.3]{GHN22}).
In this case, we drop the ``weak'' above.
For example, the cases of Coxeter type include the case of certain unitary groups of signature $(1, n-1)$ studied in \cite{VW11} by Vollaard-Wedhorn, which has been used in \cite{KR11} and \cite{Zhang12}.

To give a conceptual way to explain the relationship between the geometry of affine Deligne-Lusztig varieties and the Bruhat-Tits building of $\J$ indicated by above examples, Chen and Viehmann \cite{CV18} introduced the $\J$-stratification.
The $\J$-strata are locally closed subsets of $\cG r$.
By intersecting each $\J$-stratum with $X_{\pc\mu}(\tau_\mu)$, we obtain the $\J$-stratification of $X_{\pc\mu}(\tau_\mu)$ (see \S\ref{J-str} for details).
In \cite{Gortz19}, G\"{o}rtz showed that the Bruhat-Tits stratification coincides with the $\J$-stratification.
In fact the Bruhat-Tits stratification is a refinement of the Ekedahl-Oort stratification (see \S\ref{ADLV} for the latter).
So the $\J$-stratification is also a refinement of the Ekedahl-Oort stratification when $(G,\mu)$ is of Coxeter type.
In \cite{Shimada4}, the author classified the cases where the $\J$-stratification of $X_{\pc\mu}(\tau_\mu)$ is a refinement of the Ekedahl-Oort stratification under the assumption that $G=\GL_n$ and $b$ is superbasic.
As a result, it turns out that such cases are characterized by ``positive Coxeter'' condition (which works for general $G$, see \S\ref{LP}).
In this paper, we study this condition for $G=\GL_n$.
We also study the geometric structure of $X_{\pc\mu}(\tau_\mu)$ for $(\GL_n,\mu)$ of positive Coxeter type.
The main results of this paper are summarized below.

\begin{thm}[See Theorem \ref{classification theorem} and Theorem \ref{geometric structure}]
\label{main thm}
Let $G=\GL_n$ and let $\mu\in \Y_+$.
Then the following assertions are equivalent.
\begin{enumerate}[(i)]
\item The pair $(\GL_n,\mu)$ is of positive Coxeter type.
\item The cocharacter $\mu$ is central or one of the following forms modulo $\Z\omega_n$:
\begin{align*}
&\omega_1,\quad \omega_{n-1},\ &(n\geq 1),\\
&\omega_1+\omega_{n-1},\quad \omega_2,\quad 2\omega_1,\quad \omega_{n-2},\quad 2\omega_{n-1},\\
& \omega_2+\omega_{n-1},\quad 2\omega_1+\omega_{n-1},\quad \omega_1+\omega_{n-2},\quad\omega_1+2\omega_{n-1},\ &(n\geq 3),\\
&\omega_3,\quad\omega_{n-3},\ &(n=6,7,8),\\
&3\omega_1,\quad 3\omega_{n-1},\ &(n=3,4,5),\\
&\omega_1+\omega_2,\quad\omega_3+\omega_4,\  &(n=5),\\
&4\omega_1,\quad \omega_1+3\omega_2,\quad 4\omega_2,\quad 3\omega_1+\omega_2, &(n=3),\\
&m\omega_1\ \text{with $m\in \Z_{>0}$},  &(n=2).
\end{align*}
\end{enumerate}
Here $\omega_k$ denotes the cocharacter of the form $(1,\ldots,1,0,\ldots,0)$ in which $1$ is repeated $k$ times.
If $\mu$ satisfies the equivalent conditions above, then $X_{\pc\mu}(\tau_\mu)$ is universally homeomorphic to a union of the product of a Deligne-Lusztig variety of Coxeter type and a finite-dimensional affine space.
Moreover, this stratification coincides with the $\J_{\tau_\mu}$-stratification of $X_{\pc\mu}(\tau_\mu)$.
\end{thm}

The condition (i) is actually a natural generalization of Coxeter type.
The classification (ii) tells us that it contains many new cases not of Coxeter type (see \cite[Theorem 1.4]{GHN22} for the classification of Coxeter type).
The description of $X_{\pc\mu}(\tau_\mu)$ for $(\GL_n,\mu)$ of positive Coxeter type can be also considered as a generalization of the Bruhat-Tits stratification.
Further, we will prove that this stratification satisfies a nice closure relation in many cases (including all minuscule cases).
This closure relation again is a generalization of the closure relation which the Bruhat-Tits stratification satisfies (see \S\ref{J-str}).

Recently, Chen-Tong \cite{CT22} introduced the weak full Hodge-Newton decomposability in the context of $p$-adic Hodge theory and studied it under the minuscule condition.
They also classified the weakly fully Hodge-Newton decomposable cases.
The classification tells us that for minuscule $\mu$, $(\GL_n,\mu)$ is weakly fully Hodge-Newton decomposable if and only if $\tau_\mu$ is superbasic or $(\GL_n,\mu)$ satisfies the equivalent conditions in Theorem \ref{main thm} (see \cite[Remark 2.13]{CT22}).
In \cite[Remark 2.16]{CT22}, they pointed out that it will be an interesting question to investigate the basic affine Deligne-Lusztig varieties associated to a weakly fully Hodge-Newton decomposable pair.
The main result of this paper contains the answer to this question.
In particular, the case of $(\GL_n, \omega_2)$ is a vast generalization of the case of $(\GL_4, \omega_2)$ in \cite{Fox22} (up to perfection).

More recently, Schremmer informed the author that there is an upcoming work by He, Schremmer and Viehmann which also aims at generalizing the fully Hodge-Newton decomposable case.
For a pair $(G,\mu)$, they define a non-negative rational number $\mathrm{depth}(G,\mu)$.
Then it is known that $(G,\mu)$ is fully Hodge-Newton decomposable if and only if $\mathrm{depth}(G,\mu)\le 1$ (cf.\ \cite[Definition 3.2]{GHN19}).
They classified the cases where $1<\mathrm{depth}(G,\mu)<2$.
The classification can be reduced to the case where $G$ is simple.
If $G$ is a simple group with $1<\mathrm{depth}(G,\mu)<2$ for some $\mu$, then $G$ is a split group over $F$ of Dynkin type $A_{n-1}$, $C_3$ or $D_5$.
If $G=\GL_n$, then their case is contained in our case (both cases completely coincide if $n\geq 6$).
They also studied some geometric properties of the corresponding Ekedahl-Oort strata in a different point of view from this paper.

We focus on the case of $\GL_n$ in this paper.
However, it is natural to expect that the same description as our result would hold for general $(G,\mu)$ of positive Coxeter type.
We hope to generalize our approach for $\GL_n$ in the future to study these cases.

The paper is organized as follows.
In \S\ref{preliminaries} we introduce the affine Deligne-Lusztig variety and some relevant notion including the length positive elements and the $\J$-stratification.
In \S\ref{classification} we summarize the main results.
We first prove that the cocharacters outside the list in Theorem \ref{main thm} are not of positive Coxeter type.
After that, we finish the proof of Theorem \ref{main thm} by case-by-case analysis starting from \S\ref{omega2 2omega1}.
The main ingredients here are the techniques in \cite{Gortz19} by G\"{o}rtz and some facts on the usual Grassmannian.

\textbf{Acknowledgments:}
The author would like to thank Ulrich G\"{o}rtz for helpful discussion and Felix Schremmer for telling him about their recent work.
The author is also grateful to his advisor Yoichi Mieda for his constant support.

This work was supported by the WINGS-FMSP program at the Graduate School of Mathematical Sciences, the University of Tokyo. 
This work was also supported by JSPS KAKENHI Grant number JP21J22427.
Part of this paper was written during a stay at Universit\"{a}t Duisburg-Essen which was supported by the JSPS Overseas Challenge Program for Young Researchers.
He would like to thank the university and the host Ulrich G\"{o}rtz for their hospitality.

\section{Preliminaries}
\label{preliminaries}
Keep the notation in \S\ref{introduction}.
From now on, we sometimes drop the adjective ``perfect'' for notational convenience even in the mixed characteristic case.
Also $\cong$ always means a universal homeomorphism.
\subsection{Notation}
\label{notation}
Let $\Phi=\Phi(G,T)$ denote the set of roots of $T$ in $G$.
We denote by $\Phi_+$ (resp.\ $\Phi_-$) the set of positive (resp.\ negative) roots distinguished by $B$.
Let $\Delta$ be the set of simple roots and $\Delta^\vee$ be the corresponding set of simple coroots.
Let $X_*(T)$ be the set of cocharacters, and let $X_*(T)_+$ be the set of dominant cocharacters.

The Iwahori-Weyl group $\tW$ is defined as the quotient $N_{G(L)}T(L)/T(\cO)$.
This can be identified with the semi-direct product $W_0\ltimes X_{*}(T)$, where $W_0$ is the finite Weyl group of $G$.
We denote the projection $\tW\rightarrow W_0$ by $p$.
Let $S\subset W_0$ denote the subset of simple reflections, and let $\tS\subset \tW$ denote the subset of simple affine reflections.
We often identify $\Delta$ and $S$.
The affine Weyl group $W_a$ is the subgroup of $\tW$ generated by $\tS$.
Then we can write the Iwahori-Weyl group as a semi-direct product $\tW=W_a\rtimes \Omega$, where $\Omega\subset \tW$ is the subgroup of length $0$ elements.
Moreover, $(W_a, \tS)$ is a Coxeter system.
We denote by $\le$ the Bruhat order on $\tW$.
For any $J\subseteq \tS$, let $^J\tW$ be the set of minimal length elements for the cosets in $W_J\backslash \tW$, where $W_J$ denotes the subgroup of $\tW$ generated by $J$.
We also have a length function $\ell\colon \tW\rightarrow \Z_{\geq 0}$ given as
$$\ell(w_0\vp^{\lambda})=\sum_{\alpha\in \Phi_+, w_0\alpha\in \Phi_-}|\langle \alpha, \lambda\rangle+1|+\sum_{\alpha\in \Phi_+, w_0\alpha\in \Phi_+}|\langle \alpha, \lambda\rangle|,$$
where $w_0\in W_0$ and $\lambda\in \Y$.

For $w\in W_a$, we denote by $\supp(w)\subseteq \tS$ the set of simple affine reflections occurring in every (equivalently, some) reduced expression of $w$.
Note that $\tau\in \Omega$ acts on $\tS$ by conjugation.
We define the $\sigma$-support $\supp_\sigma(w\tau)$ of $w\tau$ as the smallest $\tau\sigma$-stable subset of $\tS$.
We call an element $w\tau\in W_a\tau$ a $\sigma$-Coxeter element if exactly one simple reflection from each $\tau\sigma$-orbit on $\supp_\sigma(w\tau)$ occurs in every (equivalently, any) reduced expression of $w$.

For $\tau\in \Omega$, let $\Sigma$ be an orbit of $\tau\sigma$ on $\tS$ and suppose that $W_\Sigma$ is finite.
We denote by $s_\Sigma$ the unique longest element of $W_\Sigma$.
Then $s_\Sigma$ is fixed by $\tau\sigma$.
The fixed point group $W_a^{\tau \sigma}\coloneqq \{w\in W_a\mid \tau\sigma(w)\tau^{-1}=w\}$ is the Weyl group whose simple reflections are the elements $s_\Sigma$ such that $\Sigma$ is a $\tau\sigma$-orbit on $\tS$ with $W_\Sigma$ finite.
%We write $\ell^{\tau \sigma}(w)$ for the length of $w$ as element of $W_a^{\tau\sigma}$.
For any reduced decomposition $w=s_{\Sigma_1}\cdots s_{\Sigma_{r}}$ as an element of $W^{\tau\sigma}_a$, we have $$\ell(w)=\ell(s_{\Sigma_1})+\cdots +\ell(s_{\Sigma_r}).$$
See \cite{Steinberg68} or \cite[\S2]{KR00} for these facts.

For $w,w'\in \tW$ and $s\in \tS$, we write $w\xrightarrow{s}_\sigma w'$ if $w'=sw\sigma(s)$ and $\ell(w')\le \ell(w)$.
We write $w\rightarrow_\sigma w'$ if there is a sequence $w=w_0,w_1,\ldots, w_k=w'$ of elements in $\tW$ such that for any $i$, $w_{i-1}\xrightarrow{s_i}_\sigma w_i$ for some $s_i\in \tS$.
If $w\rightarrow_\sigma w'$ and $w'\rightarrow_\sigma w$, we write $w\approx_\sigma w'$.

For $\alpha\in \Phi$, let $U_\alpha\subseteq G$ denote the corresponding root subgroup.
We set $$I=T(\cO)\prod_{\alpha\in \Phi_+}U_{\alpha}(\vp\cO)\prod_{\beta\in \Phi_-}U_{\beta}(\cO)\subseteq G(L),$$
which is called the standard Iwahori subgroup associated to the triple $T\subset B\subset G$.

In the case $G=\GL_n$, we will use the following description.
Let $\chi_{ij}$ be the character $T\rightarrow \Gm$ defined by $\mathrm{diag}(t_1,t_2,\ldots, t_n)\mapsto t_i{t_j}^{-1}$.
Then we have $\Phi=\{\chi_{ij}\mid i\neq j\}$, $\Phi_+=\{\chi_{ij}\mid i< j\}$, $\Phi_-=\{\chi_{ij}\mid i> j\}$ and $\Delta=\{\chi_{i,i+1}\mid 1\le i <n\}$.
Through the isomorphism $X_*(T)\cong \Z^n$, ${X_*(T)}_+$ can be identified with the set $\{(m_1,\cdots, m_n)\in \Z^n\mid m_1\geq \cdots \geq m_n\}$.
Let us write $s_1=(1\ 2), s_2=(2\ 3), \ldots, s_{n-1}=(n-1\ n)$.
Set $s_0=\vp^{\chi_{1,n}^{\vee}}(1\ n)$, where $\chi_{1,n}$ is the unique highest root.
Then $S=\{s_1,s_2,\ldots, s_{n-1}\}$ and $\tS=S\cup\{s_0\}$.
The Iwahori subgroup $I\subset K$ is the inverse image of the lower triangular matrices under the projection $G(\cO)\rightarrow G(\aFq),\  \vp\mapsto 0$.
Set $\tau={\begin{pmatrix}
0 & \vp \\
1_{n-1} & 0\\
\end{pmatrix}}$.
We often regard $\tau$ as an element of $\tW$, which is a generator of $\Omega\cong \Z$.
Note that $b\in \GL_n(L)$ is superbasic if and only if $[b]=[\tau^m]$ in $B(\GL_n)$ for some $m$ coprime to $n$.

\subsection{Affine Deligne-Lusztig Varieties}
\label{ADLV}
For $w\in \tW$ and $b\in G(L)$, the affine Deligne-Lusztig variety $X_w(b)$ in the affine flag variety $\cF l=\cF l_G=G(L)/I$ is defined as
$$X_w(b)=\{xI\in G(L)/I\mid x^{-1}b\sigma(x)\in IwI\}.$$
For $\mu\in \Y_+$ and $b\in G(L)$, the affine Deligne-Lusztig variety $X_{\mu}(b)$ in the affine Grassmannian $\cG r=\cG r_G=G(L)/K$ is defined as
$$X_{\mu}(b)=\{xK\in \cG r\mid x^{-1}b\sigma(x)\in K\vp^{\mu}K\}.$$
The closed affine Deligne-Lusztig variety is the closed reduced $\aFq$-subscheme of $\cG r$ defined as
$$X_{\pc\mu}(b)=\bigcup_{\mu'\pc \mu}X_{\mu'}(b).$$
Left multiplication by $g^{-1}\in G(L)$ induces an isomorphism between $X_\mu(b)$ and $X_\mu(g^{-1}b\sigma(g))$.
Thus the isomorphism class of the affine Deligne-Lusztig variety only depends on the $\sigma$-conjugacy class of $b$.
Moreover, we have $X_\mu(b)=X_{\mu+\ld}(\vp^{\ld}b)$ for each central $\ld \in \Y$.

The admissible subset of $\tW$ associated to $\mu$ is defined as
$$\Adm(\mu)=\{w\in \tW\mid w\le \vp^{w_0\mu}\ \text{for some}\ w_0\in W_0\}.$$
Note that $\Adm(\mu')\subseteq \Adm(\mu)$ if $\mu'\pc \mu$ (see \cite[Lemma 4.5]{Haines01}).
Set $\SAdm(\mu)=\Adm(\mu)\cap \SW$.
Then, by \cite[Theorem 3.2.1]{GH15} (see also \cite[\S2.5]{GHR20}), we have
$$X_{\pc\mu}(b)=\bigsqcup_{w\in\SAdm(\mu)}\pi(X_w(b)),$$
where $\pi\colon G(L)/I\rightarrow G(L)/K$ is the projection.
This is the so-called Ekedahl-Oort stratification.
In the sequel, we set $\SAdm(\mu)_0\coloneqq\{w\in \SAdm(\mu)\mid X_w(\tau_\mu)\neq \emptyset\}$, where $\tau_\mu\in \Omega$ such that $[\tau_\mu]\in B(G,\mu)$ is the unique basic element.

The following proposition is the key to the explicit description of the affine Deligne-Lusztig varieties.
\begin{prop}
\label{spherical}
Let $\tau\in \Omega$.
Let $w\in W_a\tau$ such that $W_{\supp_\sigma(w)}$ is finite.
Then $$X_w(\tau)=\bigsqcup_{j\in \J_\tau/\J_\tau\cap P_{\supp_\sigma(w)}} jY(w),$$
where $P_{\supp_\sigma(w)}$ is the parahoric subgroup corresponding to $\supp_\sigma(w)$ and $Y(w)=\{gI\in P_{\supp_\sigma(w)}/I\mid g^{-1}\tau \sigma(g)\in IwI\}$ is a classical Deligne-Lusztig variety in the finite-dimensional flag variety $P_{\supp_\sigma(w)}/I$.
\end{prop}
\begin{proof}
See \cite[Proposition 2.2.1]{GH15}.
\end{proof}

\subsection{The $\J$-stratification}
\label{J-str}
For any $g,h\in G(L)$, let $\inv(g,h)$ (resp.\ $\inv_K(g,h)$) denote the relative position, i.e., the unique element in $\tW$ (resp.\ $\Y_+$) such that $g^{-1}h\in I\inv(g,h)I$ (resp.\ $K\vp^{\inv_K(g,h)} K$).
By definition, two elements $gI,hI\in \cF l$ (resp.\ $gK,hK\in \cG r$) lie in the same $\J$-stratum if and only if for all $j\in \J$, $\inv(j,g)=\inv(j,h)$ (resp.\ $\inv_K(j,g)=\inv_K(j,h)$).
Clearly, this does not depend on the choice of $g,h$.
By \cite[Theorem 2.10]{Gortz19}, the $\J$-strata are locally closed in $\cF l$ (resp.\ $\cG r$).
By intersecting each $\J$-stratum with (closed) affine Deligne-Lusztig varieties, we obtain the $\J$-stratification of them.

As explained in \cite[Remark 2.1]{CV18}, the $\J$-stratification heavily depends on the choice of $b$ in its $\sigma$-conjugacy class.
So we need to fix a specific representative to compare the $\J$-stratification on $X_{\pc\mu}(b)$ to another stratification.
It is pointed out that loc. cit if $b$ is basic, then a reasonable choice is the unique length $0$ element $\tau_\mu$ in $B(G,\mu)$.
Also, for any $w\in \tW$, the $J_{\dot w}(F)$-stratification is independent of the choice of lift $\dot w$ in $G(L)$ (cf.\  \cite[Lemma 2.5]{Gortz19}).
In the rest of this subsection, we fix $b=\tau_\mu$ and hence $\J=\J_{\tau_\mu}$.
In the case of $G=\GL_n$, we set $\J^0=\{j\in \J\mid \kappa(j)=0\}$.

In general, the $\J$-stratification is too complicated to study.
However, there are several cases such that the $\J$-stratification coincides with other well-known stratifications.
Here we briefly recall two of such cases.

In the case where $G=\GL_n$ and $b=\tau^m$ with $m$ coprime to $n$, there is a group-theoretic way to describe the $\J$-stratification, which we will call the semi-module stratification.
Indeed, by \cite[Remark 3.1 \& Proposition 3.4]{CV18}, the $\J$-stratification on $\cG r$ coincides with the stratification
$$\cG r=\bigsqcup_{\ld\in \Y} I \vp^\ld K/K.$$
So in this case, each $\J$-stratum of $X_\mu(b)$ (resp.\ $X_{\pc\mu}(b)$) coincides with $X_\mu^\ld(b)$ (resp.\ $X_{\pc\mu}^\ld(b)$) for some $\ld\in \Y$, where $X_\mu^\ld(b)=X_\mu(b)\cap I \vp^\ld K/K$ (resp.\ $X_{\pc\mu}^\ld(b)=X_{\pc\mu}(b)\cap I \vp^\ld K/K$).
Note that $\tau X_\mu^\ld(b)=X_\mu^{\tau\ld}(b)$, $\J^0=\J \cap K=\J \cap I$ and  $\J/\J^0=\{\tau^k \J^0\mid k\in \Z\}$.
Thus $$\J X_\mu^\ld(b)=\bigsqcup_{k\in \Z} X_\mu^{\tau^k\ld}(b)\quad \text{and}\quad \J X_{\pc\mu}^\ld(b)=\bigsqcup_{k\in \Z} X_{\pc\mu}^{\tau^k\ld}(b).$$

Following \cite[\S 2.3]{GHN22}, we say that $(G,\mu)$ is of Coxeter type if $$\SAdm(\mu)_0=\{w\in \SAdm(\mu)\mid \text{$W_{\supp_\sigma(w)}$ is finite, and $w$ is $\sigma$-Coxeter in $W_{\supp_\sigma(w)}$}\}.$$
If $(G,\mu)$ is of Coxeter type, then for each $w\in \SAdm(\mu)_0$, we have
$$\pi(X_w(\tau_\mu))=\bigsqcup_{j\in \J/\J\cap P_w}j \pi(Y(w)),$$
where $Y(w)$ is the same as in Proposition \ref{spherical}, and $P_w$ is the parahoric subgroup corresponding to $\supp(w)\cup \max \{ J\subseteq S\mid \Ad(w)(\sigma(J))=J\}$ (cf.\ \cite[Proposition 5.7]{GHN19}).
Moreover, we have $Y(w)\cong \pi(Y(w))$.
Thus if $(G,\mu)$ is of Coxeter type, we obtain the decomposition $X_{\pc \mu}(\tau_\mu)$ as a union of classical Deligne-Lusztig varieties of Coxeter type in a natural way.
We call this stratification the {\it Bruhat-Tits stratification}.
Also, this is a stratification in the strong sense, i.e., the closure of a stratum is a union of strata.
The closure of a stratum $j\pi(Y(w))$ contains a stratum $j'\pi(Y(w'))$ if and only if the following two conditions are both satisfied:
\begin{enumerate}[(1)]
\item $w\geq_{S,\sigma}w'$, which means by definition that there exists $u\in W_0$ such that $w\geq u^{-1}w'\sigma(u)$.
\item $j(\J\cap P_w)\cap j'(\J\cap P_{w'})\neq \emptyset$.
\end{enumerate}
By \cite[\S 4.7]{He07}, $\geq_{S,\sigma}$ gives a partial order on $\SW$.
Let $\cB(\J, F)$ denote the rational Bruhat-Tits building of $\J$.
Then (2) above is equivalent to requiring that $\kappa(j)=\kappa(j')$ and that the simplices in $\cB(\J, F)$ corresponding to $j(\J\cap P_w)j^{-1}$ and $j'(\J\cap P_{w'})j'^{-1}$ are neighbors (i.e., there exists an alcove which contains both of them).
In \cite{Gortz19}, G\"ortz proved that the Bruhat-Tits stratification coincides with the $\J$-stratification.

Unlike the Bruhat-Tits stratification, the $\J$-stratification always exists.
So the $\J$-stratification is expected to play an important role to go beyond the cases of Coxeter type.
In this paper, we will treat cases such that the $\J$-stratification of $X_{\pc \mu}(b)$ is a refinement of the Ekedahl-Oort stratification.
More precisely, we treat cases satisfying all of the following conditions:
\begin{itemize}
\item For $w\in \SAdm(\mu)_0$, $\J$ acts transitively on the set of irreducible components of $X_w(\tau_\mu)$.
\item For $w\in \SAdm(\mu)_0$, there exist a parahoric subgroup $P_w\subset G(L)$ and an irreducible component $Y(w)$ of $X_w(\tau_\mu)$ such that $\pi(X_w(\tau_\mu))=\bigsqcup_{j\in \J/\J\cap P_w}j \pi(Y(w))$.
\item $Y(w)\cong \pi(Y(w))$ and each $j \pi(Y(w))$ is a $\J$-stratum of $X_{\pc \mu}(b)$.
\end{itemize}
In these cases, we say that the closure relation can be described in terms of $\cB(\J, F)$ if the $\J$-stratification of $X_{\pc \mu}(b)$ is a stratification in the strong sense and $\overline{j\pi(Y(w))}\supseteq j'\pi(Y(w'))$ is equivalent to the following condition:
\begin{enumerate}[\null]
\item There exist sequences $w=w_0\geq_{S,\sigma} w_1\geq_{S,\sigma}\cdots \geq_{S,\sigma}w_k=w'$ in $\SAdm(\mu)_0$ and $j=j_0, j_1\ldots, j_k=j'$ in $\J$ such that $j_{i-1}(\J\cap P_{w_{i-1}})\cap j_i(\J\cap P_{w_i})\neq \emptyset$ for $1\le i\le k$.
\end{enumerate}

\subsection{Deligne-Lusztig Reduction Method}
\label{DL method}
The following Deligne-Lusztig reduction method was established in \cite[Corollary 2.5.3]{GH10}.
\begin{prop}
\label{DL method prop}
Let $w\in \tW$ and let $s\in \tS$ be a simple affine reflection.
If $\ch(F)>0$, then the following two statements hold for any $b\in G(L)$.
\begin{enumerate}[(i)]
\item If $\ell(sw\sigma(s))=\ell(w)$, then there exists a $\J_b$-equivariant universal homeomorphism $X_w(b)\rightarrow X_{sw\sigma(s)}(b)$.
\item If $\ell(sw\sigma(s))=\ell(w)-2$, then there exists a decomposition $X_w(b)=X_1\sqcup X_2$ such that
\begin{itemize}
\item $X_1$ is open and there exists a $\J_b$-equivariant morphism $X_1\rightarrow X_{sw}(b)$, which is  the composition of a Zariski-locally trivial $\G_m$-bundle and a universal homeomorphism. 
\item $X_2$ is closed and there exists a $\J_b$-equivariant morphism $X_2\rightarrow X_{sw\sigma(s)}(b)$, which is the composition of a Zariski-locally trivial $\A^1$-bundle and a universal homeomorphism. 
\end{itemize}
If $\ch(F)=0$, then the above statements still hold by replacing $\A^1$ and $\G_m$ by $\A^{1,\pfn}$ and $\G_m^{\pfn}$ respectively.
\end{enumerate}
\end{prop}

Let $gI\in X_w(b)$.
If $\ell(sw)<\ell(w)$ (we can reduce to this case by exchanging $s$ and $sw\sigma(s)$), then let $g_1I$ denote the unique element in $\cF l$ such that $\inv(g, g_1)=s$ and $\inv(g_1,b\sigma(g))=sw$.
The set $X_1$ (resp.\ $X_2$) above consists of the elements $gI\in X_w(b)$ satisfying $\inv(g_1,b\sigma(g_1))=sw$ (resp.\ $sw\sigma(s)$).
All of the maps in the proposition are given as the map sending $gI$ to $g_1I$.
\begin{rema}
\label{trivial}
Assume that $G$ is split over $F$.
Let $\pr_I\colon \cF l\times \cF l\rightarrow \cF l$ be the projection to the first factor.
We denote by $O(s)\subset \cF l\times \cF l$ the locally closed subvariety of pairs $(gI,hI)$ such that $\inv(g,h)=s$.
Then the restriction $\pr_I\colon O(s)\rightarrow \cF l$ is a Zariski-locally trivial $\A^1$-bundle.
More precisely, this is trivial over (any translation of) the ``big open cell'' (cf.\  \cite[pp.\ 45--48]{Faltings03}).
In particular, this is trivial over any Schubert cell $IvI/I, v\in \tW$.
This implies that the morphism $X_1\rightarrow X_{sw}(b)$ (resp.\ $X_2\rightarrow X_{sw\sigma(s)}(b)$) in (ii) is trivial over $X_{sw}(b)\cap IvI/I$ (resp.\ $X_{sw\sigma(s)}(b)\cap IvI/I$).
\end{rema}

The following result is proved in \cite[Theorem 2.10]{HN14}, which allows us to reduce the study of $X_w(b)$ for any $w$, via the Deligne-Lusztig reduction method, to the study of $X_w(b)$ for $w$ of minimal length in its $\sigma$-conjugacy class.
\begin{theo}
\label{minimal}
For each $w\in \tW$, there exists an element $w'$ which is of minimal length inside its $\sigma$-conjugacy class such that $w\rightarrow_\sigma w'$.
\end{theo}

Following \cite[\S 3.4]{HNY22}, we construct the reduction trees for $w$ by induction on $\ell(w)$.

The vertices of the trees are elements of $\tW$.
We write $x\rightharpoonup y$ if $x,y\in \tW$ and there exists $x'\in \tW$ and $s\in \tS$ such that $x\approx_\sigma x'$, $\ell(sx'\sigma(s))=\ell(x')-2$ and $y\in \{sx', sx'\sigma(s)\}$.
These are the (oriented) edges of the trees.

If $w$ is of minimal length in its $\sigma$-conjugacy class of $\tW$, then the reduction tree for $w$ consists of a single vertex $w$ and no edges.
Assume that $w$ is not of minimal length and that a reduction tree is given for any $z\in \tW$ with $\ell(z)<\ell(w)$.
By Theorem \ref{minimal}, there exist $w'$ and $s\in \tS$ with $w\approx_\sigma w'$ and $\ell(sw'\sigma(s))=\ell(w')-2$.
Then a reduction tree of $w$ consists of the given reduction trees of $sw'$ and $sw'\sigma(s)$ and the edges $w\rightharpoonup sw'$ and $w\rightharpoonup sw'\sigma(s)$. 

Let $\cT$ be a reduction tree of $w$.
An end point of $\cT$ is a vertex in $\cT$ of minimal length.
A reduction path in $\cT$ is a path $\unp\colon w \rightharpoonup w_1\rightharpoonup\cdots \rightharpoonup w_n$, where $w_n$ is an end point of $\cT$.
Set $\en(\unp)=w_n$.
We say that $x\rightharpoonup y$ is of type I (resp.\ II) if $\ell(x)-\ell(y)=1$ (resp.\ $\ell(x)-\ell(y)=2$).
For any reduction path $\unp$, we denote by $\ell_{I}(\unp)$ (resp.\ $\ell_{II}(\unp)$) the number of type I (resp.\ II) edges in $\unp$.
We write $X_{\unp}$ a locally closed subscheme of $X_w(b)$ which is $\J_b$-equivariant universally homeomorphic to an iterated fibration of type $(\ell_{I}(\unp),\ell_{II}(\unp))$ over $X_{\en(\unp)}(b)$.

Let $B(\tW,\sigma)$ be the set of $\sigma$-conjugacy classes in $\tW$.
Let $\Psi:B(\tW,\sigma)\rightarrow B(G)$ be the map sending $[w]\in B(\tW,\sigma)$ to $[\dot w]\in B(G)$.
It is known that this map is well-defined and surjective, see \cite[Theorem 3.7]{He14}.
By Proposition \ref{DL method prop} (cf.\ \cite[Proposition 3.9]{HNY22}), we have the following description of $X_w(b)$.
\begin{prop}
\label{decomposition}
Let $w\in \tW$ and $\cT$ be a reduction tree of $w$.
For any $b\in G(L)$, there exists a decomposition
$$X_w(b)=\bigsqcup_{\substack{\unp\ \text{is a reduction path in}\ \cT;\\ \Psi(\en(\unp))=[b]}}X_{\unp}.$$
\end{prop}

\subsection{Length Positive Elements}
\label{LP}
We denote by $\delta^+$ the indicator function of the set of positive roots, i.e.,
$$\delta^+\colon \Phi\rightarrow \{0,1\},\quad \alpha \mapsto
\begin{cases}
1 & (\alpha\in \Phi_+) \\
0 & (\alpha\in \Phi_-).
\end{cases}$$
Note that any element $w\in \tW$ can be written in a unique way as $w=x\vp^\mu y$ with $\mu$ dominant, $x,y\in W_0$ such that $\vp^\mu y\in \SW$.
We have $p(w)=xy$ and $\ell(w)=\ell(x)+\la\mu, 2\rho\ra-\ell(y)$.
We define the set of {\it length positive} elements by $$\LP(w)=\{v\in W_0\mid \la v\alpha,y^{-1}\mu\ra+\delta^+(v\alpha)-\delta^+(xyv\alpha)\geq 0\  \text{for all $\alpha\in \Phi_+$}\}.$$
Then we always have $y^{-1}\in \LP(w)$.
Indeed, $y$ is uniquely determined by the condition that
$\la\alpha, \mu\ra\geq \delta^+(-y^{-1}\alpha)\ \text{for all $\alpha\in \Phi_+$}$.
Since $\delta^+(\alpha)+\delta^+(-\alpha)=1$, we have $$\la y^{-1}\alpha, y^{-1}\mu\ra+\delta^+(y^{-1}\alpha)-\delta^+(x\alpha)=\la \alpha,\mu\ra-\delta^+(-y^{-1}\alpha)+\delta^+(-x\alpha)\geq 0.$$
\begin{lemm}
\label{LPr}
For any $w=x\vp^\mu y\in \tW$ as above, we define $$\Phi_w\coloneqq\{\alpha\in \Phi_+\mid \la\alpha,\mu\ra-\delta^-(y^{-1}\alpha)+\delta^-(x\alpha)=0\}.$$
Here $\delta^-$ denotes the indicator function of the set of negative roots. Then we have
$$y\LP(w)=\{r^{-1}\in W_0\mid r(\Phi_+\setminus \Phi_w)\subset \Phi_+\ \text{or equivalently,}\ r^{-1}\Phi_+\subset \Phi_+\cup-\Phi_w\}.$$
\end{lemm}
\begin{proof}
See \cite[Lemma 2.8]{Shimada4}.
\end{proof}

The notion of length positive elements is defined by Schremmer in \cite{Schremmer22}.
The description of $\LP(w)$ in Lemma \ref{LPr} is due to Lim \cite{Lim23}.

We say that the Dynkin diagram of $G$ is $\sigma$-connected if it cannot be written as a union of two proper $\sigma$-stable subdiagrams that are not connected to each other. 
The following theorem is a refinement of the non-emptiness criterion in \cite{GHN15}, which is conjectured by Lim in \cite{Lim23} and proved by Schremmer in \cite[Proposition 5]{Schremmer23}.
\begin{theo}
\label{empty}
Assume that the Dynkin diagram of $G$ is $\sigma$-connected.
Let $b\in G(L)$ be a basic element with $\kappa(b)=\kappa(w)$.
Then $X_w(b)=\emptyset$ if and only if the following two conditions are satisfied:
\begin{enumerate}[(i)]
\item $|W_{\supp_\sigma(w)}|$ is not finite.
\item There exists $v\in \LP(w)$ such that $\supp_\sigma(\sigma^{-1}(v)^{-1}p(w)v)\subsetneq S$.
\end{enumerate}
\end{theo}

\begin{rema}
If $\kappa(b)\neq\kappa(w)$, then $X_w(b)=\emptyset$.
\end{rema}

For $w\in \tW$, we say that $w$ has positive Coxeter part if there exists $v\in \LP(w)$ such that $v^{-1}p(w)v$ is a Coxeter element.
Although not needed in this paper, it is worth mentioning that this condition induces a simple geometric structure of $X_w(b)$.
\begin{theo}
\label{simple}
Assume that $w\in \tW$ has positive Coxeter part and $X_w(b)\neq \emptyset$.
Then $X_w(b)$ has only one $\J_b$-orbit of irreducible components, and each irreducible component is an iterated fibration over a Deligne-Lusztig variety of Coxeter type whose iterated fibers are either $\A^1$ or $\G_m$.
If $b$ is basic, then all fibers are $\A^1$.
\end{theo}
\begin{proof}
This is \cite[Theorem A]{SSY23}.
See also \cite[Theorem 1.1]{HNY22}.
\end{proof}

For $\mu\in \Y$, we say that $(G,\mu)$ is of positive Coxeter type if every $w\in \SAdm(\mu)_0$ has positive Coxeter part.
If $(G,\mu)$ is of positive Coxeter type, then by Theorem \ref{simple}, $X_{\pc \mu}(\tau_\mu)$ should have a simple geometric structure.
If $G=\GL_n$, it follows from \cite[Theorem 4.12]{SSY23} that every minimal length element in its $\sigma$-conjugacy class has positive Coxeter part.
Thus if $(\GL_n,\mu)$ is of Coxeter type, then $(\GL_n,\mu)$ is of positive Coxeter type.

\section{Classification}
\label{classification}
From now and until the end of this paper, we set $G=\GL_n$ and $b=\tau^m$.
Fix $\mu\in \Y_+$.
Let $\mu(i)$ denote the $i$-th entry of $\mu$.
Then $[\tau^m]\in B(G,\mu)$ if and only if $m=\mu(1)+\cdots+\mu(n)$.
We assume this from now (i.e., $b=\tau^m=\tau_\mu$).

\subsection{The Ekedahl-Oort Stratification for $\omega_2+\omega_{n-2}$}
Throughout this subsection, we assume $\mu=\omega_2+\omega_{n-2}$.
Also we assume that $n\geq 4$.

\begin{lemm}
\label{EO 2n-2 not}
There exists $w\in \SAdm(\mu)^\circ\coloneqq \SAdm(\mu)\setminus \SAdm(\omega_1+\omega_{n-1})$ such that $v^{-1}p(w)v$ is not a Coxeter element for any $v\in \LP(w)$ and $X_w(b)\neq \emptyset$.
\end{lemm}
\begin{proof}
First assume that $n\geq 6$.
Set $y=s_{n-2}s_{n-1}s_{n-3}\cdots s_2s_1s_3=(1\ n-1\ n\ n-2\ n-3\ \cdots\ 2)(3\ 4)=(1\ n-1\ n\ n-2\ \cdots\ 5\ 4\ 2)$.
It is straightforward to check that $\vp^\mu y\in \SAdm(\mu)^\circ$ and $\chi_{1,3},\chi_{3,n}\in \Phi_+\setminus \Phi_{\vp^\mu y}$.
Let $r\in W_0$.
If $ryr^{-1}\in \bigcup_{J\subsetneq S}W_J$, then $(r(1)\ r(n-1)\ r(n)\ r(n-2)\ \cdots\ r(5)\ r(4)\ r(2))=(1\ 2\ \cdots\ n-1)$ or $(2\ 3\ \cdots\ n)$.
In the former (resp.\ latter) case, we must have $r(3)=n$ (resp.\ $r(3)=1$) and hence $r\chi_{3,n}$ (resp.\ $r\chi_{1,3}$) is negative.
Then by Lemma \ref{LPr} and Theorem \ref{empty}, the statement holds for $w=\vp^\mu y$.
If $n=4$ (resp.\ $5$), set $y=s_2s_3s_1s_2=(1\ 3)(2\ 4)$ (resp.\ $s_3s_4s_2s_1s_3s_2=(1\ 4\ 2\ 5\ 3)$).
Then it is easy to check that the statement holds for $w=\vp^\mu y$.
\end{proof}

\subsection{Classification}
Let $\omega_k$ denote the cocharacter of the form $(1,\ldots,1,0,\ldots,0)$ in which $1$ is repeated $k$ times.
The following theorem is the classification of Coxeter type for $\GL_n$ (cf.\ \cite[Theorem 1.4]{GHN22}).
\begin{theo}
The following assertions on $\mu$ are equivalent.
\begin{enumerate}[(i)]
\item The pair $(\GL_n,\mu)$ is of Coxeter type.
\item The cocharacter $\mu$ is central or one of the following forms modulo $\Z\omega_n$:
\begin{align*}
\omega_1,\quad \omega_{n-1}\ (n\geq 1),\quad \omega_1+\omega_{n-1}\ (n\geq 2),\quad \omega_2\ (n=4).
\end{align*}
\end{enumerate}
\end{theo}

The following theorem is the classification of positive Coxeter type for $\GL_n$.
See \cite[Theorem 7.2]{Shimada4} for the superbasic case.
\begin{theo}
\label{classification theorem}
The following assertions on $\mu$ are equivalent.
\begin{enumerate}[(i)]
\item The pair $(\GL_n,\mu)$ is of positive Coxeter type.
\item The cocharacter $\mu$ is central or one of the following forms modulo $\Z\omega_n$:
\begin{align*}
&\omega_1,\quad \omega_{n-1},\ &(n\geq 1),\\
&\omega_1+\omega_{n-1},\quad \omega_2,\quad 2\omega_1,\quad \omega_{n-2},\quad 2\omega_{n-1},\\
& \omega_2+\omega_{n-1},\quad 2\omega_1+\omega_{n-1},\quad \omega_1+\omega_{n-2},\quad\omega_1+2\omega_{n-1},\ &(n\geq 3),\\
&\omega_3,\quad\omega_{n-3},\ &(n=6,7,8),\\
&3\omega_1,\quad 3\omega_{n-1},\ &(n=3,4,5),\\
&\omega_1+\omega_2,\quad\omega_3+\omega_4,\  &(n=5),\\
&4\omega_1,\quad \omega_1+3\omega_2,\quad 4\omega_2,\quad 3\omega_1+\omega_2, &(n=3),\\
&m\omega_1\  \text{with $m\in \Z_{>0}$}, &(n=2).
\end{align*}
\end{enumerate}
\end{theo}

\begin{proof}
As explained in the last paragraph of \cite[\S 2.5]{Shimada4}, it is enough to treat one of $\mu$ or $-w_{\max}\mu$ for the implication (ii) $\Rightarrow$ (i).
This follows from \cite[\S 6 \& Theorem 7.2]{Shimada4} and the case-by-case analysis starting from the next section.
So we only prove the implication (i) $\Rightarrow$ (ii) here.

Let $0\le m_0<n$ be the residue of $m$ modulo $n$.
If $m_0=0$ and $n\geq 4$, then $\omega_2+\omega_{n-2}+(\mn-1)\omega_n\pc \mu$ unless $\mu=\omega_0$ or $\omega_1+\omega_{n-1}$ modulo $\Z \omega_n$.
So if $\mu\neq \omega_0,\omega_1+\omega_{n-1}$ modulo $\Z\omega_n$, then $\mu$ does not satisfy (i) by Lemma \ref{EO 2n-2 not}.
If $m_0=0$ and $n=3$, then $2\omega_1+2\omega_2+(\mn-2)\omega_n\pc \mu$ unless $\mu= \omega_0, \omega_1+\omega_2, 3\omega_1$ or $3\omega_2$ modulo $\Z\omega_n$.
It follows from Lemma \ref{empty} that $\vp^{2\omega_1+2\omega_2+(\mn-2) \omega_n}(1\ 3)\in \SAdm(\mu)$ and $X_{\vp^{2\omega_1+2\omega_2+\mn \omega_n}(1\ 3)}(b)\neq \emptyset$.
So if $\mu\neq \omega_0, \omega_1+\omega_2, 3\omega_1, 3\omega_2$ modulo $\Z\omega_n$, then $\mu$ does not satisfy (i).
If $m_0=2$ and $n=4$, then $3\omega_2+(\mn-1)\omega_n\pc \mu$ unless $\mu=\omega_2,2\omega_1$ or $2\omega_3$ modulo $\Z\omega_n$.
It follows from Lemma \ref{empty} that $\vp^{3\omega_2+(\mn-1)\omega_n}(1\ 3)(2\ 4)\in \SAdm(\mu)$ and $X_{3\omega_2+(\mn-1)\omega_n}(b)\neq \emptyset$.
So if $\mu\neq\omega_2,2\omega_1, 2\omega_3$ modulo $\Z\omega_n$, then $\mu$ does not satisfy (i).
If $m_0=2$ and $n=6$, then $\omega_3+\omega_5+(\mn-1)\omega_n\pc \mu$ unless $\mu=\omega_2$ or $2\omega_1$ modulo $\Z\omega_n$.
It is easy to check that $s_0s_1s_2s_5\tau^8\in \SAdm(\omega_3+\omega_5)$, $X_{s_0s_1s_2s_5\tau^8}(\tau^8)\neq \emptyset$ and $p(s_0s_5s_4s_1\tau^4)=(1\ 6\ 3)(2\ 4\ 5)$.
So if $\mu\neq\omega_2,2\omega_1$ (resp.\ $\mu\neq\omega_4,2\omega_5$) modulo $\Z\omega_n$, then $\mu$ does not satisfy (i).
Other cases are superbasic and follow from (the proof of) \cite[\S 6 \& Theorem 7.2]{Shimada4}.
Note that the proof there literally works for non-superbasic $b$.
\end{proof}

It is easy to check that the image of $X_\mu(b)$ under the automorphism of $\cG r$ for $G=\GL_n$ by $gK\mapsto w_{\max}{^tg}^{-1}K$ is $X_{-w_{\max}\mu}(b^{-1})$.
Also this automorphism maps a $\J_b$-stratum of $\cG r$ to a $\J_{b^{-1}}$-stratum of $\cG r$.
Thus to study the geometric structure, it is enough to treat one of $\mu$ or $-w_{\max}\mu$ modulo $\Z\omega_n$.
Thus the following theorem follows from the cases of Coxeter type (cf.\ \S\ref{J-str}), \cite[\S 6 \& Theorem 7.2]{Shimada4} and the case-by-case analysis starting from the next section.
\begin{theo}
\label{geometric structure}
Assume that $\mu$ satisfies the equivalent conditions in Theorem \ref{classification theorem}.
Then the following assertions hold:
\begin{itemize}
\item For $w\in \SAdm(\mu)_0$, $\J$ acts transitively on the set of irreducible components of $X_w(b)$.
\item For $w\in \SAdm(\mu)_0$, there exist a parahoric subgroup $P_w\subset G(L)$ and an irreducible component $Y(w)$ of $X_w(b)$ such that $\pi(X_w(b))=\bigsqcup_{j\in \J/\J\cap P_w}j \pi(Y(w))$.
\item Each $j \pi(Y(w))$ is a $\J$-stratum of $X_{\pc \mu}(b)$.
\item $Y(w)\cong \pi(Y(w))$ is universally homeomorphic to the product of a Deligne-Lusztig variety of Coxeter type and a finite-dimensional affine space.
\end{itemize}
Assume moreover that $\mu\neq \omega_2+\omega_{n-1}, 2\omega_1+\omega_{n-1}, \omega_1+\omega_{n-2}, \omega_1+2\omega_{n-1}$.
Then the closure relation can be described in terms of $\cB(\J,F)$.
\end{theo}

It is natural to expect that the closure relation can be also described in terms of $\cB(\J,F)$ if $\mu= \omega_2+\omega_{n-1}, 2\omega_1+\omega_{n-1}, \omega_1+\omega_{n-2}, \omega_1+2\omega_{n-1}$.
Although $b$ is superbasic in all of these cases, the number of $\J$-strata is large and hence the closure relation is too complicated to study.

\section{The cases of $\omega_2$ and $2\omega_1$}
\label{omega2 2omega1}
Keep the notation above.
In this section, we set $b=\tau^2$.
Then the $F$-rank of $\J=\J_b$ is $0$ (resp.\ $1$) if $n$ is odd (resp.\ even).

\subsection{The case of $\omega_2$ when $n$ is odd}
In this subsection, we set $\mu=\omega_2$.
Assume that $n\geq 5$ is odd, i.e., $b$ is superbasic.
For any $1\le k\le \frac{n-3}{2}(=\dim X_{\mu}(b))$, set
$$\ld_k=
\begin{cases}
\chi_{1,n}^\vee+\chi_{3,n-2}^\vee+\cdots+\chi_{k,n-k+1}^\vee & (k\ \text{odd})\\
\chi_{2,n-1}^\vee+\chi_{4,n-3}^\vee+\cdots+\chi_{k,n-k+1}^\vee & (k\ \text{even}).
\end{cases}$$
We also set $\ld_0=\omega_0$.
Let $X_\mu(b)^{i}=\{gK\in X_\mu(b)\mid \kappa(g)=i\}$.
Then $X_\mu(b)^{i}$ is a closed subvariety with $X_\mu(b)=\bigsqcup_{i\in \Z} X_\mu(b)^{i}=\bigsqcup_{i\in \Z} \tau^iX_\mu(b)^0$.
\begin{prop}
\label{EO 2 odd}
We have
$$\SAdm(\mu)_0=\{\tau^2, s_0s_{n-1}\tau^2,s_0s_{n-1}s_{n-2}s_{n-3}\tau^2,\ldots, s_0s_{n-1}\cdots s_5s_4\tau^2\}.$$
Every element of $\SAdm(\mu)_0$ has positive Coxeter part.
For $0\le k\le \frac{n-3}{2}$, let $w_k$ denote the unique element in $\SAdm(\mu)_0$ of length $2k$.
Then there exists an irreducible component $Y(w_k)$ of $X_{w_k}(b)$ such that $X_{w_k}(b)=\J Y(w_k)$, $Y(w_k)\cong\pi(Y(w_k))=X_{\mu}^{\ld_k}(b)\cong \A^k$ and
$$\pi(X_{w_k}(b))=\bigsqcup_{j\in \J/\J\cap I}j X_{\mu}^{\ld_k}(b).$$
Moreover, the closure relation can be described in terms of $\cB(\J,F)$.
\end{prop}
\begin{proof}
Except the ``moreover'' part, the proposition follows from \cite[Proposition 6.5]{Shimada4} and \cite[Theorem 5.3]{Viehmann08}.
It is enough to check the closure relation in $X_\mu(b)^{i}$ for some $i$.
Since $\tau^{\frac{n+3}{2}}\ld_k\in W_0\omega_{\frac{n+3}{2}}$ (resp.\ $\tau^{\frac{n+5}{2}}\ld_k\in W_0\omega_{\frac{n+5}{2}}$) for $1\le k\le \frac{n-3}{2}$ (resp.\ $1\le k\le \frac{n-5}{2}$), we have 
$$X_\mu(b)^{\frac{n+3}{2}}\subset \cG r(\omega_{\frac{n+3}{2}})\quad \text{(resp.\ $(X_\mu(b)^{\frac{n+5}{2}}\setminus X_\mu^{\tau^{\frac{n+5}{2}}\ld_{\frac{n-3}{2}}}(b))\subset \cG r(\omega_{\frac{n+5}{2}})$),} $$
where $\cG r(\ld)=K\vp^\ld K/K$ for $\ld\in \Y\cong \Z^n$.

To show the closure relation, we argue by induction on odd $n$.
If $n=5$, then this follows from the equidimensionality of $X_\mu(b)$.
Assume that the closure relation can be described in terms of $\cB(\J,F)$ for $5,7,\ldots, n-2$.
Again by the equidimensionality of $X_\mu(b)$, the closure of $X_\mu^{\ld_{\frac{n-3}{2}}}$ is $X_\mu(b)^0$.
Let $K'=\GL_{n-2}(\cO)$ and let $I'$ be the standard Iwahori subgroup in it.
For $\ld'\in \Z^{n-2}$, let $\cG r'(\ld')$ denote $K'\vp^{\ld'} K'/K'$.
Let $\omega'_k=(1,\ldots,1,0,\ldots,0)\in \Z^{n-2}$ in which $1$ is repeated $k$ times.
We define $\iota\colon \cG r'(\omega'_{\frac{n+1}{2}})\rightarrow \cG r(\omega_{\frac{n+5}{2}})$ by $g'K'\mapsto \begin{pmatrix}
\vp^{(1,1)} & 0 \\
0 & g'\\
\end{pmatrix}K$.
Clearly this map is well-defined and (universally) injective.
Since $\cG r'(\omega'_{\frac{n+1}{2}})$ and $\cG r(\omega_{\frac{n+5}{2}})$ are projective over $\aFq$, $\iota$ is a universal homeomorphism onto its image (in fact, it is easy to check that $\iota$ is a monomorphism and hence a closed immersion).
Set $\mu'=\omega'_2\in \Z^{n-2}$, $\tau'={\begin{pmatrix}
0 & \vp \\
1_{n-3} & 0\\
\end{pmatrix}}\in \GL_{n-2}(F)$ and $b'=\tau'^2$.
Let $X_{\mu'}(b')^{\frac{n+1}{2}}\subset \cG r'(\omega'_{\frac{n+1}{2}})\subset \cG r_{\GL_{n-2}}$.
For $0\le k\le \frac{n-5}{2}$, we define $\ld'_k\in \Z^{n-2}$ by $\tau^{\frac{n+5}{2}}\ld_k=(1,1,\ld'_k)$.
Then $X_{\mu'}(b')^{\frac{n+1}{2}}=\bigsqcup_{0\le k\le \frac{n-5}{2}}X_{\mu'}^{\ld'_k}(b')$.
Note that the first three entries of $\ld'_k$ are $1$.
So for any $g'K'\in X_{\mu'}^{\ld'_k}(b')$ ($gK\in X_{\mu}^{\ld_k}(b)$), there exists $$h'\in \begin{pmatrix}
1_2 & 0 \\
0 & \GL_{n-4}(L)\\
\end{pmatrix}\cap I'\quad \text{(resp.\ $h\in \begin{pmatrix}
1_4 & 0 \\
0 & \GL_{n-4}(L)\\
\end{pmatrix}\cap I$ )}$$such that $g'K'=h'\vp^{\ld'_k}K'$ (resp.\ $gK=h\vp^{\ld_k}K$).
By $b=\dot s_2\dot s_3\dot s_1\dot s_2\begin{pmatrix}
1_2 & 0 \\
0 & b'\\
\end{pmatrix}$, we have
 $$\vp^{-\ld_k}\begin{pmatrix}
1_2 & 0 \\
0 & h'^{-1}\\
\end{pmatrix}b\begin{pmatrix}
1_2 & 0 \\
0 & \sigma(h')\\
\end{pmatrix}\vp^{\ld_k}=\dot s_2\dot s_3\dot s_1\dot s_2\begin{pmatrix}
1_2 & 0 \\
0 & \vp^{-\ld'_k}h'^{-1}b'\sigma(h'\vp^{\ld'_k})\\
\end{pmatrix}\in K\vp^\mu K,$$
where $\dot s_i$ denotes the permutation matrix corresponding to $s_i$.
Conversely, for $h$ above, we define $h'$ by $h=\begin{pmatrix}
1_2 & 0 \\
0 & h'\\
\end{pmatrix}$.
Then we can similarly check that $h'\vp^{\ld'_k}K'\in X_{\mu'}^{\ld'_k}(b')$.
Thus $\iota(X_{\mu'}^{\ld'_k}(b'))=X_{\mu}^{\tau^{\frac{n+5}{2}}\ld_k}(b)$ for $0\le k\le \frac{n-5}{2}$.
It follows from this and the induction hypothesis that the closure relation can be described in terms of $\cB(\J,F)$ for $n$.
By induction, this finishes the proof.
\end{proof}

\subsection{The case of $\omega_2$ when $n$ is even}
\label{EO 2 even}
In this subsection, we set $\mu=\omega_2$.
Then the $F$-rank of $\J$ is $1$.
Assume that $n\geq 4$ is even.

\begin{lemm}
\label{EO 2 empty}
We have
$$\SAdm(\mu)_0=\{\tau^2, s_0\tau^2,s_0s_{n-1}s_{n-2}\tau^2,\ldots, s_0s_{n-1}\cdots s_5s_4\tau^2\}.$$
Every element of $\SAdm(\mu)_0$ has positive Coxeter part.
\end{lemm}
\begin{proof}
Set $u_1=s_2s_3\cdots s_{n-1}s_1$.
For $2\le k\le n-2$, we also set $u_k=u_1s_2s_3\cdots s_k$.
Then $\vp^\mu u_k\in \SAdm(\mu)$.
Moreover it is easy to check that $w\in \SAdm(\mu)$ satisfies $\supp_{\sigma}(w)\neq \tS$ if and only if $w$ is $\vp^\mu u_{n-2}=\tau^2$ or $\vp^\mu u_{n-3}=s_0\tau^2$.
So it follows from Theorem \ref{empty} that $\SAdm(\mu)_0\subseteq \{\vp^\mu u_1,\ldots,\vp^\mu u_{n-2}\}$.

For even $k$, we have $$u_k=(1\ 3\ \cdots\ k+1)(2\ 4\ \cdots\ k+2\ k+3\ \cdots\ n-1\  n).$$
For odd $k$, we have $$u_k=(1\ 3\ \cdots\ k+2\ k+3\ \cdots\ n\ 2\ 4\ \cdots\ k+1).$$
It is easy to check that $\Phi_+\setminus \Phi_{\vp^\mu u_k}=\{\chi_{1,k+3},\ldots,\chi_{1,n-1},\chi_{1,n}\}.$
In particular, $r(\Phi_+\setminus \Phi_{\vp^\mu u_k})\subset \Phi_+$ for $r\in W_{\{s_2,\ldots,s_{n-1}\}}$.
Again from Theorem \ref{empty}, it follows that $X_{\vp^\mu u_k}(b)=\emptyset$ for even $2\le k\le n-2$.
We can also check that other $\vp^\mu u_k$ has positive Coxeter part.
This finishes the proof.
\end{proof}

\begin{coro}
\label{EO 2 Gm}
Set $v_{k,l}=s_{n-l}s_{n-l-1}\cdots s_{n-2k+l+1}\tau^2$ for $2\le k\le \frac{n-2}{2}$ and $1\le l\le k-1$.
Then $X_{v_{k,l}}(b)=\emptyset$.
\end{coro}
\begin{proof}
We keep the notation in the proof of Lemma \ref{EO 2 empty}.
It is easy to check that $v_{k,l}\approx_\sigma s_0s_{n-1}\cdots s_{n-2(k-l)+1}\tau^2=\vp^\mu u_{n-2(k-l-1)}$.
Then the statement follows from Proposition \ref{DL method prop} and Lemma \ref{EO 2 empty}.
\end{proof}

For $w\in \SW$, set $S_w=\max\{S'\subseteq S\mid \Ad(w)(S')=S'\}$.
Clearly $S_{\tau^2}=\{s_1,s_3,\ldots,s_{n-1}\}$.
\begin{coro}
\label{EO 2 Sw}
Let $\vp^\mu y\in \SAdm(\mu)\setminus \{\tau^2\}$ such that $\ell(y)\geq n-1$.
Then $S_{\vp^\mu y}=\emptyset$.
\end{coro}
\begin{proof}
We keep the notation in the proof of Lemma \ref{EO 2 empty}.
Then $y=u_k$ for some $1\le k\le n-3$, and the lemma follows from explicit computation.
\end{proof}

There are two $\tau^2$-orbit in $\tS$, namely, $\{s_0,s_2,\ldots, s_{n-2}\}$ and $\{s_1,s_3,\ldots, s_{n-1}\}$.
Let $P_{0}$ (resp.\ $P_1$) denote the standard parahoric subgroup of $G(L)$ corresponding to the former (resp.\ latter) orbit.
Then $W_a^{\tau^2}$ is the Weyl group with two simple reflections $s_0s_2\cdots s_{n-2}$ and $s_1s_3\cdots s_{n-1}$.
For $j,j'\in \J^0$, we have $\inv(j,j')\in W_a^{\tau^2}$.

\begin{lemm}
\label{EO 2 Iwahori}
For $1\le k\le \frac{n-2}{2}$, let $w_k$ denote the unique element in $\SAdm(\mu)_0$ of length $2k-1$ (more precisely, $w_k=s_0s_{n-1}\cdots s_{n-2(k-1)}\tau^2$).
For any $w_k$, there exists an irreducible component $Y(w_k)$ of $X_{w_k}(b)$ such that $$X_{w_k}(b)=\bigsqcup_{j\in\J/\J\cap P_{w_k}}jY(w_k),$$
where $P_{w_k}=P_0$ (resp.\ $P_1$) if $k$ is odd (resp.\ even).
Moreover each $Y(w_k)$ is universally homeomorphic to $(\mathbb P^1\setminus \mathbb P^1(\F_{q^{\frac{n}{2}}}))\times \A^{k-1}$, and contained in a $\J$-stratum in $G(L)/I$.
\end{lemm}
\begin{proof}
For an integer $a$, let $0\le [a]<n$ denote its residue modulo $n$.
For $a,b\in \N$ with $a-b\in 2\Z$, we define
$t_{a,b}=s_{[b-2]}\cdots s_{[a+2]}s_{[a]}$ ($t_{a,b}=1$ if $a-b\in n\Z$).
Set
\begin{align*}
w_{k,0}=w_k,w_{k,1}=t_{0,n-2(k-1)}w_kt_{0,n-2(k-1)}^{-1}, w_{k,2}=t_{n-1, n-2k+3}w_{k,2}t_{n-1, n-2k+3}^{-1},\\
\ldots, w_{k,k-1}=t_{n-k+2,n-k}w_{k,k-2}t_{n-k+2,n-k}^{-1}.
\end{align*}
Clearly (the simple reflections in the support of) $t_{0,n-2(k-1)},t_{n-1, n-2k+3},\ldots, t_{n-k+2,n-k}$ define
\begin{align*}
w_k=w_{k,0}\rightarrow_{\sigma}w_{k,1}=s_{n-1}s_{n-2}\cdots s_{n-2k+3}\tau^2&\rightarrow_{\sigma}w_{k,2}=s_{n-2}s_{n-3}\cdots s_{n-2k+4}\tau^2\\
&\rightarrow_{\sigma}\cdots \rightarrow_{\sigma} w_{k,k-1}=s_{n-k+1}\tau^2.
\end{align*}
Let $\unp_k$ be the reduction path (in a suitable reduction tree) defined by this reduction.
By Proposition \ref{decomposition} and Corollary \ref{EO 2 Gm}, we have $X_{w_k}(b)=X_{\unp_k}$ with $\en(\unp_k)=s_{n-k+1}\tau^2$.
Let $f\colon X_{w_k}(b)\rightarrow X_{s_{n-k+1}\tau^2}(b)$ be the morphism induced by Proposition \ref{DL method prop}.
By Proposition \ref{spherical}, we have $$X_{s_{n-k+1}\tau^2}(b)=\bigsqcup_{j\in \J/\J\cap P_{w_k}} jY(s_{n-k+1}\tau^2),$$
where  $Y(s_{n-k+1}\tau^2)=\{gI\in P_{w_k}/I\mid g^{-1}\tau^2 \sigma(g)\tau^{-2}\in Is_{n-k+1}I\}$ is a classical Deligne-Lusztig variety in the finite-dimensional flag variety $P_{w_k}/I$.
If $k$ is odd (resp.\ even), set $v_k=s_0s_2\cdots s_{n-2}$ (resp.\ $s_1s_3\cdots s_{n-1}$).
Then by \cite[Corollary 2.5]{Lusztig76} (see also \cite[Proposition 1.1]{Gortz19}), $Y(s_{n-k+1}\tau^2)$ is contained in $Iv_kI/I$.
Moreover, it is easy to check that 
\begin{align*}
&\ell(v_kt_{n-k+2,n-k}\cdots t_{n-1,n-2k+3}t_{0,n-2(k-1)})\\
=&\frac{n}{2}+\ell(t_{n-k+2,n-k})+\cdots+ \ell(t_{n-1,n-2k+3})+\ell(t_{0,n-2(k-1)}).
\end{align*}
Indeed $v_kt_{n-k+2,n-k}\cdots t_{n-k+l+1,n-k-l+1}\alpha_i=v_{k}v_{k+1}\cdots v_{k+l}\alpha_i$ for $l\in \Z_{\geq 0}$, where $\alpha_i$ is the simple affine root corresponding to $i\in \supp(t_{n-k+l+2,n-k-l})$.
Since $W_a^{\tau^2}$ is the Weyl group with two simple reflections $v_k$ and $v_{k+1}$, $v_{k}v_{k+1}\cdots v_{k+l}\alpha_i>0$.
We set $Y(w_k)=f^{-1}(Y(s_{n-k+1}\tau^2))$.
Then by Proposition \ref{DL method prop}, we have 
\begin{align*}
Y(w_k)\subset Iv_kt_{n-k+2,n-k}\cdots t_{n-1,n-2k+3}t_{0,n-2(k-1)}I/I.
\end{align*}
Note that $(\tau^2)^{\frac{n}{2}}=\vp^{(1,\ldots,1)}$ belongs to the center of $G(L)$.
It follows from this fact that $Y(s_{n-k+1}\tau^2)\cong \mathbb P^1\setminus \mathbb P^1(\F_{q^{\frac{n}{2}}})$.
Thus by Remark \ref{trivial}, $Y(w_k)$ is universally homeomorphic to $(\mathbb P^1\setminus \mathbb P^1(\F_{q^{\frac{n}{2}}}))\times \A^{k-1}$.

It remains to show that for all $j\in \J$, the value $\inv(j,-)$ is constant on each $j'Y(w_k)$.
For this, we argue similarly as \cite[\S 3.3]{Gortz19}.
Clearly we may assume $j'=1$.
For any $j\in \J$, there exists $\tilde j\in \J^0$ such that $\inv(j, \tilde j)\in\Omega$.
So we may also assume $j\in \J^0$.
Fix $jI$ with $j\in \J^0$.
Then by \cite[Corollary 2.5]{Lusztig76} and \cite[Proposition 5.34]{AB08} (see also \cite[Proposition 1.7]{Gortz19}), there exists $gI$ with $g\in \J^0\cap P_{w_k}$ (called the ``gate'') such that for any $y_0I\in Y(s_{n-k+1}\tau^2)$, we have
\begin{align*}
\inv(j, y_0)=\inv(j,g)v_k(\in W_a)\quad \text{with}\quad \ell(\inv(j, y_0))=\ell(\inv(j,g))+\ell(v_k).
\end{align*}
In particular, $\inv(j, g)$ has a reduced expression as an element of $W_a^{\tau^2}$ whose rightmost simple reflection is $v_{k+1}$.
Let $y\in Y(w_k)$ and set $y_0=f(y)\in Y(s_{n-k+1}\tau^2)$.
Note that 
\begin{align*}
&\ell(\inv(j, y_0)t_{n-k+2,n-k}\cdots t_{n-1,n-2k+3}t_{0,n-2(k-1)})\\
=&\ell(\inv(j, g))+\ell(v_k)+\ell(t_{n-k+2,n-k})+\cdots+ \ell(t_{n-1,n-2k+3})+\ell(t_{0,n-2(k-1)}).
\end{align*}
Indeed $k$ is odd (resp.\ even) if and only if $n-k$ is odd (resp.\ even).
Thus
\begin{align*}
\inv(j, y)&=\inv(j, y_0)\inv(y_0,y)\\
&=\inv(j,g)v_kt_{n-k+2,n-k}\cdots t_{n-1,n-2k+3}t_{0,n-2(k-1)}
\end{align*}
is independent of $y\in Y(w_k)$.
This finishes the proof.
\end{proof}

For any $1\le k\le \frac{n-2}{2}(=\dim X_{\mu}(b))$, set
$$\ld_k=
\begin{cases}
\chi_{1,n}^\vee+\chi_{3,n-2}^\vee+\cdots+\chi_{k,n-k+1}^\vee & (k\ \text{odd})\\
\chi_{2,n-1}^\vee+\chi_{4,n-3}^\vee+\cdots+\chi_{k,n-k+1}^\vee & (k\ \text{even}).
\end{cases}$$
We also set $w_0=\tau^2,P_{w_0}=P_1, Y(w_0)=\{pt\}$ and $\ld_0=\omega_0$.
\begin{prop}
\label{EO 2 even}
Keep the notation above.
Then $$\pi(X_{w_k}(b))=\bigsqcup_{j\in \J/\J\cap P_{w_k}}j \pi(Y(w_k))$$
and each $j\pi(Y(w_k))$ is a $\J$-stratum of $X_{\mu}(b)$, which is universally homeomorphic to $(\mathbb P^1\setminus \mathbb P^1(\F_{q^{\frac{n}{2}}}))\times \A^{k-1}$ (resp.\ $\{pt\}$) if $1\le k\le \frac{n-2}{2}$ (resp.\ $k=0$).
Moreover, the closure relation can be described in terms of $\cB(\J,F)$.
\end{prop}
\begin{proof}
Let $1\le k\le \frac{n-2}{2}$.
We first prove that
$$\pi^{-1}(\pi(X_{w_k}(b)))\cap (\bigcup_{w'\le w_k}X_{w'}(b))=X_{w_k}(b)$$
and hence $X_{w_k}(b)$ is closed in $\pi^{-1}(\pi(X_{w_k}(b)))$.
We have $\pi(X_w(b))\cap \pi(X_{w_k}(b))=\emptyset$ for $w\in \SAdm(\mu)$ with $\ell(w)< \ell(w_k)$.
By Proposition \ref{spherical}, we also have $\pi(X_w(b))=\pi(X_{\tau^2}(b))$ and hence $\pi(X_w(b))\cap \pi(X_{w_k}(b))=\emptyset$ for any $w\in W_{\{s_1,s_3,\ldots,s_{n-1}\}}\tau^2=W_{S_{\tau^2}}\tau^2$.
Note that $\SW\cap W_0\vp^\mu W_0=\SAdm(\mu)$.
Then the above equality follows from \cite[Proposition 3.1.1]{GH15}, Proposition \ref{DL method} and Corollary \ref{EO 2 Sw}.
Using \cite[Lemma 2.1]{Shimada4}, we can easily check that the map $X_{w_k}(b)\rightarrow \pi(X_{w_k}(b))$ induced by $\pi$ is universally bijective.
Since $\pi$ is proper, the map $X_{w_k}(b)\rightarrow \pi(X_{w_k}(b))$ is also proper.
This implies that the map $X_{w_k}(b)\rightarrow \pi(X_{w_k}(b))$  is a universal homeomorphism.
In particular, $\pi(Y(w_k))$ is universally homeomorphic to $(\mathbb P^1\setminus \mathbb P^1(\F_{q^{\frac{n}{2}}}))\times \A^{k-1}$.
Clearly, the same assertions are true for each $j\pi(Y(w))$.
The case $k=0$ is trivial.

We next prove the closure relation.
It follows from the proof of Lemma \ref{EO 2 Iwahori} that $\pi(Y(w_k))\subset X_\mu^{\ld_k}(b)$ for any $0\le k\le \frac{n-2}{2}$.
Since $\tau^{\frac{n+2}{2}}\ld_k\in W_0 \omega_{\frac{n+2}{2}}$, we have $$\tau^{\frac{n+2}{2}}\pi(Y(w_k))\subset X_\mu^{\tau^{\frac{n+2}{2}}\ld_k}(b)\subset \cG r(\omega_{\frac{n+2}{2}}),$$
where $\cG r(\ld)=K\vp^\ld K/K$ for $\ld\in \Y\cong \Z^n$.

Let $K'=\GL_{n+1}(\cO)$ and let $I'$ be the standard Iwahori subgroup in it.
For $\ld'\in \Z^{n+1}$, let $\cG r'(\ld')$ denote $K'\vp^{\ld'} K'/K'$.
Let $\omega'_k=(1,\ldots,1,0,\ldots,0)\in \Z^{n+1}$ in which $1$ is repeated $k$ times.
We define $\iota\colon \cG r(\omega_{\frac{n+2}{2}})\rightarrow \cG r'(\omega'_{\frac{n+4}{2}})$ by $gK\mapsto \begin{pmatrix}
\vp & 0 \\
0 & g\\
\end{pmatrix}K'$.
Clearly this map is well-defined and (universally) injective.
Since $\cG r(\omega_{\frac{n+2}{2}})$ and $\cG r'(\omega'_{\frac{n+4}{2}})$ are projective over $\aFq$, $\iota$ is a universal homeomorphism onto its image (in fact, it is easy to check that $\iota$ is a monomorphism and hence a closed immersion).
For $0\le k\le \frac{n-2}{2}$, we define $\ld'_k\in \Z^{n+1}$ by $\ld'_k=(1,\tau^{\frac{n+2}{2}}\ld_k)$.
Set $\mu'=\omega'_2\in \Z^{n+1}$, $\tau'={\begin{pmatrix}
0 & \vp \\
1_{n} & 0\\
\end{pmatrix}}\in \GL_{n+1}(F)$ and $b'=\tau'^2$.
Then $X_{\mu'}(b')^{\frac{n+4}{2}}=\bigsqcup_{0\le k\le \frac{n-2}{2}}X_{\mu'}^{\ld'_k}(b')$, where $X_{\mu'}(b')^{i}=\{g'K'\in X_{\mu'}(b')\mid \kappa(g')=i\}$.
Similarly as the proof of Proposition \ref{EO 2 odd}, we can check that $\iota(X_\mu^{\tau^{\frac{n+2}{2}}\ld_k}(b))=X_{\mu'}^{\ld'_k}(b')$ for $0\le k\le \frac{n-2}{2}$.
By Proposition \ref{EO 2 odd} and $\dim \tau^{\frac{n+2}{2}}\pi(Y(w_k))=\dim X_{\mu'}^{\ld'_k}(b')=k$, we have $$\overline{\tau^{\frac{n+2}{2}}\pi(Y(w_k))}=\bigsqcup_{0\le k'\le k}X_\mu^{\tau^{\frac{n+2}{2}}\ld_{k'}}(b),\quad \text{or equivalently,}\quad \overline{\pi(Y(w_k))}=\bigsqcup_{0\le k'\le k}X_\mu^{\ld_{k'}}(b).$$

We follow the notation in the proof of Lemma \ref{EO 2 Iwahori}.
Let $1\le k\le \frac{n-2}{2}$, $j\in \J^0$ and $l\in \Z_{\geq 0}$.
Then by the proof of Lemma \ref{EO 2 Iwahori}, we have 
\begin{align*}
jY(w_k)\subset Iv_{k+l}\cdots v_{k+2}v_{k+1}v_kt_{n-k+2,n-k}\cdots t_{n-1,n-2k+3}t_{0,n-2(k-1)}I/I.
\end{align*}
Indeed, by replacing $j$ by another representative in $j(\J\cap P_{w_k})$ if necessary, we may assume $\inv(1,j)=v_{k+l}\cdots v_{k+2}v_{k+1}$ for some $l$ (we set $\inv(1,j)$=1 if $l=0$).
Thus we can explicitly compute $\ld\in\Y$ such that $j\pi(Y(w_k))\subset X_\mu^\ld(b)$.
Combining this with the description of each $\overline{\pi(Y(w_k))}$ above, we can easily verify the closure relation ($(\J\cap P_{w_k})\cap j(\J\cap P_{w_{k-1}})\neq \emptyset$ is equivalent to $\inv(1, j)=1$ or $v_k$).

It remains to show that each $j\pi(Y(w_k))$ is a $\J$-stratum of $X_{\mu}(b)$.
By Lemma \ref{EO 2 Iwahori}, $j\pi(Y(w_k))$ is contained in a $\J$-stratum.
Note that each $\J$-stratum of $X_\mu(b)$ is contained in $X_\mu^\ld(b)$ for some $\ld\in \Y$.
So in particular, by an explicit computation as above, the strata $j\pi(Y(w_k))$ for fixed $k$ are contained in different $\J$-strata from each other.
To finish the proof, we need to show that for $j,j'\in \J^0$ and $k>k'$ such that $j\pi(Y(w_k)),j'\pi(Y(w_{k'}))\subset X_\mu^\ld(b)$ for some $\ld\in \Y$, $j\pi(Y(w_k))$ and $j'\pi(Y(w_{k'}))$ are contained in different $\J$-strata.
We may assume that $j=1$.
Let $y\in Y(w_k)$ and set $y_0=f(y)\in Y(s_{n-k+1}\tau^2)$ as in the proof of Lemma \ref{EO 2 Iwahori}.
Then there exists $gI$ with $g\in \J^0\cap P_{w_k}$ such that 
\begin{align*}
\inv(j', y_0)=\inv(j',g)v_k(\in W_a)\quad \text{with}\quad \ell(\inv(j', y_0))=\ell(\inv(j',g))+\ell(v_k).
\end{align*}
Thus $\inv(j',Y(w_{k}))=\inv(j', y)=\inv(j',g)v_kt_{n-k+2,n-k}\cdots t_{n-1,n-2k+3}t_{0,n-2(k-1)}$.
This also implies that $\inv_K(j', \pi(Y(w_k)))\neq \inv_K(j', j'\pi(Y(w_{k'})))$.
This finishes the proof.
\end{proof}

\subsection{The case of $2\omega_1$ when $n$ is odd}
In this subsection, we set $\mu=2\omega_1$.
Note that the unique dominant cocharacter $\mu'$ with $\mu'\prec \mu$ is $\mu'=\omega_2$.
Assume that $n\geq 3$ is odd, i.e., $b$ is superbasic.
For any $1\le k\le \frac{n-1}{2}(=\dim X_{\pc\mu}(b))$, set
$$\ld_k=
\begin{cases}
\chi_{1,n}^\vee+\chi_{3,n-2}^\vee+\cdots+\chi_{k,n-k+1}^\vee & (k\ \text{odd})\\
\chi_{2,n-1}^\vee+\chi_{4,n-3}^\vee+\cdots+\chi_{k,n-k+1}^\vee & (k\ \text{even}).
\end{cases}$$
We also set $\ld_0=\omega_0$.
\begin{prop}
We have
$$\SAdm(\mu)_0=\{\tau^2, s_0s_{n-1}\tau^2,s_0s_{n-1}s_{n-2}s_{n-3}\tau^2,\ldots, s_0s_{n-1}\cdots s_3s_2\tau^2\}.$$
Every element of $\SAdm(\mu)_0$ has positive Coxeter part.
For $0\le k\le \frac{n-1}{2}$, let $w_k$ denote the unique element in $\SAdm(\mu)_0$ of length $2k$.
Then there exists an irreducible component $Y(w_k)$ of $X_{w_k}(b)$ such that $X_{w_k}(b)=\J Y(w_k)$, $Y(w_k)\cong\pi(Y(w_k))=X_{\mu}^{\ld_k}(b)\cong \A^k$ and
$$\pi(X_{w_k}(b))=\bigsqcup_{j\in \J/\J\cap I}j X_{\mu}^{\ld_k}(b).$$
Moreover, the closure relation can be described in terms of $\cB(\J,F)$.
\end{prop}
\begin{proof}
The case $0\le k\le \frac{n-3}{2}$ follows from Proposition \ref{EO 2 odd}.
Clearly $w_{\frac{n-1}{2}}=s_0s_{n-1}\cdots s_3s_2\tau^2=\vp^\mu s_1s_2\cdots s_{n-1}$ has positive Coxeter part.
The description of $\pi(X_{w_{\frac{n-1}{2}}}(b))$ follows similarly as the proof of \cite[Proposition 6.5]{Shimada4} (cf.\ Remark \ref{trivial}).
The closure relation follows from Proposition \ref{EO 2 odd}, $\dim \pi(X_{w_{\frac{n-1}{2}}}(b))=\dim X_{\pc\mu}(b)$ and the equidimensionality of $X_{\pc \mu}(b)$.
\end{proof}

\subsection{The case of $2\omega_1$ when $n$ is even}
In this subsection, we set $\mu=2\omega_1$.
Note that the unique dominant cocharacter $\mu'$ with $\mu'\prec \mu$ is $\mu'=\omega_2$.
Assume that $n\geq 4$ is even.
\begin{prop}
We have
$$\SAdm(\mu)_0=\{\tau^2, s_0\tau^2,s_0s_{n-1}s_{n-2}\tau^2,\ldots, s_0s_{n-1}\cdots s_3s_2\tau^2\}.$$
Every element of $\SAdm(\mu)_0$ has positive Coxeter part.
For $1\le k\le \frac{n}{2}$, let $w_k$ denote the unique element in $\SAdm(\mu)_0$ of length $2k-1$ (more precisely, $w_k=s_0s_{n-1}\cdots s_{n-2(k-1)}\tau^2$).
Set $w_0=\tau^2$ and $Y(w_0)=\{pt\}$.
For any $w_k$, there exists an irreducible component $Y(w_k)$ of $X_{w_k}(b)$ such that $X_{w_k}(b)=\J Y(w_k)$, $Y(w_k)\cong \pi(Y(w_k))$ and $$\pi(X_{w_k}(b))=\bigsqcup_{j\in \J/\J\cap P_{w_k}}j \pi(Y(w_k)),$$
where $P_{w_k}=P_0$ (resp.\ $P_1$) if $k$ is odd (resp.\ even).
Each $\pi(Y(w_k))$ is a $\J$-stratum of $X_{\mu}(b)$, which is universally homeomorphic to $(\mathbb P^1\setminus \mathbb P^1(\F_{q^{\frac{n}{2}}}))\times \A^{k-1}$ if $1\le k\le \frac{n}{2}$.
Moreover, the closure relation can be described in terms of $\cB(\J,F)$.
\end{prop}
\begin{proof}
The case where $0\le k\le \frac{n-2}{2}$ follows from Proposition \ref{EO 2 even}.
Clearly $w_{\frac{n}{2}}=s_0s_{n-1}\cdots s_3s_2\tau^2=\vp^\mu s_1s_2\cdots s_{n-1}$ has positive Coxeter part.
The description of $\pi(Y(w_{\frac{n}{2}}))$ follows similarly as the proof of Lemma \ref{EO 2 Iwahori} and Proposition \ref{EO 2 even}.

To finish the proof, it remains to check the closure relation.
Set
$$\ld_{\frac{n}{2}}=
\begin{cases}
\chi_{1,n}^\vee+\chi_{3,n-2}^\vee+\cdots+\chi_{k,n-k+1}^\vee & (\frac{n}{2}\ \text{odd})\\
\chi_{2,n-1}^\vee+\chi_{4,n-3}^\vee+\cdots+\chi_{k,n-k+1}^\vee & (\frac{n}{2}\ \text{even}).
\end{cases}$$
Then $\pi(Y(w_{\frac{n}{2}}))\subset I\vp^{\ld_{\frac{n}{2}}}K/K$ for $Y(w_{\frac{n}{2}})$ defined similarly as the proof of Lemma \ref{EO 2 Iwahori}.
Thus we can explicitly compute $\ld\in\Y$ such that $j\pi(Y(w_k))\subset X_\mu^\ld(b)$.
It easily follows from this that if $(j, w_k)$ satisfies $j\pi(Y(w_k))\cap\overline{\pi(Y(w_{\frac{n}{2}}))}\neq \emptyset$, then $(j, w_k)$ satisfies the closure relation in \S\ref{J-str}.
In particular, if $(\J\cap P_{w_{\frac{n}{2}}})\cap j(\J\cap P_{w_{\frac{n-2}{2}}})=\emptyset$, then $j\pi(Y(w_{\frac{n-2}{2}}))\cap \overline{\pi(Y(w_{\frac{n}{2}}))}=\emptyset$.
Since $X_{\pc \mu}(b)$ is equidimensional and $\overline{\pi(Y(w_{\frac{n}{2}}))}$ is an irreducible component, we have $X_{\pc \mu}(b)=\bigcup_{j\in \J/\J\cap P_{w_{\frac{n}{2}}}} j\overline{\pi(Y(w_{\frac{n}{2}}))}$.
So by the irreducibility of $\pi(Y(w_{\frac{n-2}{2}}))$, there exists $j_0\in \J^0$ such that $\pi(Y(w_{\frac{n-2}{2}}))\subset j_0^{-1}\overline{\pi(Y(w_{\frac{n}{2}}))}$, or equivalently, $j_0\pi(Y(w_{\frac{n-2}{2}}))\subset \overline{\pi(Y(w_{\frac{n}{2}}))}$.
Note that this $j_0$ must satisfy $(\J\cap P_{w_{\frac{n}{2}}})\cap j_0(\J\cap P_{w_{\frac{n-2}{2}}})\neq \emptyset$.
Thus by multiplying $j_0\pi(Y(w_{\frac{n-2}{2}}))$ by elements in $\J\cap P_{w_{\frac{n}{2}}}$, we deduce that if $(\J\cap P_{w_{\frac{n}{2}}})\cap j(\J\cap P_{w_{\frac{n-2}{2}}})\neq \emptyset$, then $j\pi(Y(w_{\frac{n-2}{2}}))\subset \overline{\pi(Y(w_{\frac{n}{2}}))}$.
It follows from these facts and Proposition \ref{EO 2 even} that the closure relation can be described in terms of $\cB(\J,F)$.
\end{proof}

\section{The cases of $\omega_3$, $\omega_1+\omega_2$ and $3\omega_1$}
Keep the notation above.
In this section, we set $b=\tau^3$.

\subsection{The case of $\omega_3$ when $n=6$}
In this subsection, we set $\mu=\omega_3$.
We assume that $n=6$.

There are three $\tau^3$-orbit in $\tS$, namely, $\{s_0,s_3\}$, $\{s_1,s_4\}$ and $\{s_2,s_5\}$.
For $i,i'\in \{0,1,\ldots,5\}$, let $P_{ii'}$ denote the standard parahoric subgroup of $G(L)$ corresponding to the union of the orbits of $s_i$ and $s_{i'}$.
Let $\Omega^2_{\F_{q^2}}$ denote the (perfection of) $2$-dimensional Drinfeld's upper half-space over $\F_{q^2}$.
Then the Deligne-Lusztig variety $\{g\in G(\aFq)/B(\aFq)\mid \inv(g, \sigma^2(g))=s_1s_2\}$ of Coxeter type is isomorphic to $\Omega^2_{\F_{q^2}}$.
\begin{prop}
\label{EO 3 even}
We have
$$\SAdm(\mu)_0=\{\tau^3,s_0\tau^3,s_0s_1\tau^3,s_0s_5\tau^3,s_0s_1s_5s_0\tau^3\}.$$
Every element of $\SAdm(\mu)_0$ has positive Coxeter part.
Set \begin{align*}
P_{\tau^2}=P_{12},\quad P_{s_0\tau^3}=P_{01}\cap P_{02},\quad P_{s_0s_1\tau^3}=P_{01},\quad P_{s_0s_5\tau^3}=P_{02},\quad P_{s_0s_1s_5s_0\tau^3}=P_{12}.
\end{align*}
For any $w\in \SAdm(\mu)_0$, there exists an irreducible component $Y(w)$ of $X_w(b)$ such that $X_w(b)=\J Y(w)$, $Y(w)\cong \pi(Y(w))$ and $$\pi(X_{w}(b))=\bigsqcup_{j\in \J/\J\cap P_{w}}j \pi(Y(w)).$$
Each $j\pi(Y(w))$ is a $\J$-stratum of $X_{\mu}(b)$ with
\begin{align*}
&\pi(Y(\tau^2))\cong \{pt\},\quad \pi(Y(s_0\tau^3))\cong \mathbb P^1\setminus \mathbb P^1(\F_{q^2}),\quad \pi(Y(s_0s_1\tau^3))\cong \Omega^2_{\F_{q^2}},\\
&\pi(Y(s_0s_5\tau^3))\cong \Omega^2_{\F_{q^2}},\quad
\pi(Y(s_0s_1s_5s_0\tau^3))\cong \Omega^2_{\F_{q^2}}\times \A^1.
\end{align*}
Moreover, the closure relation can be described in terms of $\cB(\J,F)$.
\end{prop}
\begin{proof}
The assertions for $w\in \SAdm(\mu)_0\setminus \{s_0s_1s_5s_0\tau^2\}$ follow from Proposition \ref{spherical}, (the proof of) \cite[Theorem 7.2.1]{GH15} and easy computation of finite part.

Set $w=s_0s_1s_5s_0\tau^3$.
It remains to describe $\pi(Y(w))$ and $\overline{\pi(Y(w))}$.
We have $$w=s_0s_1s_5s_0\tau^3\xrightarrow{s_0}_\sigma s_3s_1s_5s_0\tau^3 \xrightarrow{s_3}_\sigma s_1s_5\tau^3.$$
By $p(w)=s_3s_4s_5s_2s_1$ and \cite[Theorem 5.1]{HNY22}, $X_{s_1s_5s_0\tau^3}(b)=\emptyset$.
Let $f\colon X_{w}(b)\rightarrow X_{s_1s_5\tau^3}(b)$ be the morphism induced by Proposition \ref{DL method prop}.
By Proposition \ref{spherical}, we have $$X_{s_1s_5\tau^3}(b)=\bigsqcup_{j\in \J/\J\cap P_{12}} jY(s_1s_5\tau^3),$$
where  $Y(s_1s_5\tau^3)=\{gI\in P_{12}/I\mid g^{-1}\tau^3 \sigma(g)\tau^{-3}\in Is_1s_5I\}$ is a classical Deligne-Lusztig variety in the finite-dimensional flag variety $P_{12}/I$.
It is easy to check that $Y(s_1s_5\tau^3)\cong \Omega^2_{\F_{q^2}}$.
Set $v=s_1s_4s_2s_5s_1s_4=s_2s_5s_1s_4s_2s_5$.
By \cite[Corollary 2.5]{Lusztig76} (see also \cite[Proposition 1.1]{Gortz19}), $Y(s_1s_5\tau^3)$ is contained in $IvI/I$.
Set $Y(w)=f^{-1}(Y(s_1s_5\tau^3))$.
Then by Proposition \ref{DL method prop}, $Y(w)\subset Ivs_3s_0I/I$.
Also by Remark \ref{trivial}, $Y(w)$ is universally homeomorphic to $\Omega^2_{\F_{q^2}}\times \A^1$.

We next show that for all $j\in \J$, the value $\inv(j,-)$ is constant on each $j'Y(w)$.
Similarly as the proof of Lemma \ref{EO 2 Iwahori}, we may assume that $j'=1$ and $j\in \J^0$.
Fix $jI$ with $j\in \J^0$.
Then by \cite[Corollary 2.5]{Lusztig76} and \cite[Proposition 5.34]{AB08}, there exists $gI$ with $g\in \J^0\cap P_{12}$ such that for any $y_0I\in Y(s_1s_5\tau^3)$, we have
\begin{align*}
\inv(j, y_0)=\inv(j,g)v(\in W_a)\quad \text{with}\quad \ell(\inv(j, y_0))=\ell(\inv(j,g))+\ell(v).
\end{align*}
In particular, $\inv(j, g)$ has a reduced expression as an element of $W_a^{\tau^3}$ whose rightmost simple reflection is $s_0s_3$ unless $\inv(j, g)=1$.
Let $y\in Y(w)$ and set $y_0=f(y)\in Y(s_1s_5\tau^3)$.
Note that 
$\ell(\inv(j, y_0)s_0s_3)
=\ell(\inv(j, g))+\ell(v)+\ell(s_0s_3)$.
Thus
\begin{align*}
\inv(j, y)=\inv(j, y_0)\inv(y_0,y)=\inv(j,g)vs_0s_3
\end{align*}
is independent of $y\in Y(w)$.
This proves the value $\inv(j,-)$ (resp.\  $\inv_K(j,-)$) is constant on each $Y(w)$ (resp.\ $\pi(Y(w))$).

We next describe $\overline{\pi(Y(w))}$ as a union of other strata.
Let $j\in \J^0$.
We will prove the following two assertions for $i=1$ or $5$:
\begin{enumerate}[(1)]
\item If $(\J\cap P_{12})\cap j(\J\cap P_{0i})\neq \emptyset$, then $j\pi(Y(s_0s_{i}\tau^3))\subset \overline{\pi(Y(w))}$.
\item Otherwise, $j\pi(Y(s_0s_i\tau^3))\cap\overline{\pi(Y(w))}=\emptyset$.
\end{enumerate}
We only treat the case $i=1$.
The proof for the case $i=5$ is similar.
By replacing $j$ by another representative in $j(\J\cap P_{01})$  if necessary, we may assume that $\inv(1,j)$ is the minimal length representative of its coset in $W_a/W_{\{s_0,s_3,s_1,s_4\}}$.
By \cite[Corollary 2.5]{Lusztig76}, $Y(s_0s_1\tau^3)$ is contained in $Is_1s_4s_0s_3s_1s_4I/I$.
So there exists $\mu_j\in \Y_+$ such that $j\pi(Y(s_0s_1\tau^3))\subset K\vp^{\mu_j}K/K$.
Note that $(\J\cap P_{12})\cap j(\J\cap P_{01})\neq \emptyset$ is equivalent to $\inv(1,j)=1,s_2s_5$ or $s_1s_4s_2s_5$.
Moreover, if $(\J\cap P_{12})\cap j(\J\cap P_{01})=\emptyset$, then $s_0$ (and hence $s_3$) belongs to $\supp(\inv(1,j))$.
Combining this with \cite[(2.7.11)]{Macdonald03}, we deduce that if $(\J\cap P_{12})\cap j(\J\cap P_{01})=\emptyset$, then $(1,1,0,0,-1,-1)\pc \mu_j$.
This proves (2).
Since $\pi(Y(s_0s_{i}\tau^3))$ is irreducible and $X_\mu(b)$ is equidimensional, there exists $j_0\in \J^0$ such that $j_0\pi(Y(s_0s_1\tau^3))\subset \overline{\pi(Y(w))}$.
This $j_0$ must satisfy $(\J\cap P_{12})\cap j(\J\cap P_{01})\neq \emptyset$.
Thus by multiplying $j_0\pi(Y(s_0s_1\tau^3))$ by elements in $\J\cap P_{12}$, we deduce that if $(\J\cap P_{12})\cap j(\J\cap P_{01})\neq \emptyset$, then $j\pi(Y(s_0s_{1}\tau^3))\subset \overline{\pi(Y(w))}$.
This proves (1).
We can also verify similar statements for $j\pi(Y(s_0\tau^3))$ and $j\pi(Y(\tau^2))$, which proves the closure relation.
By computing $\J$-invariants explicitly, we can also verify that each $j\pi(Y(w))$ is a $\J$-stratum (cf.\ the proof of Lemma \ref{EO 2 even}).
This finishes the proof.
\end{proof}

\subsection{The case of $\omega_3$ when $n=7,8$}
In this subsection, we set $\mu=\omega_3$.
The following is a well-known fact:
\begin{lemm}
\label{codim=1}
If $X$ is a (separated) variety and $U\subsetneq X$ is an affine open subscheme, then $\dim X\setminus U=\dim X-1$.
\end{lemm}

For $d\le n$, every $d\times n$ matrix $A$ of rank $d$ with coefficients in $\aFq$ determines an element $W_A$ in the Grassmannian $\mathrm{Gr}_d(\aFq^n)$ such that $W_A$ is generated by the $d$ rows of $A$.
Let $I_{d,n}$ be the set of subsets of $d$ elements in $\{1,2,\ldots,n\}$.
For any $J\in I_{d,n}$, let $A_J$ be the $d\times d$ matrix whose columns are the columns of $A$ at indices from $J$.
Let $L_A=\{J\in I_{d,n}\mid \det A_J\neq 0\}$.
In fact, $L_A$ only depends on $W_A$, and we can associate with every element $W$ of $\mathrm{Gr}_d(\aFq^n)$ a corresponding set $L_W$.
This is called the {\it list} of $W$.
The locally closed subsets of $\mathrm{Gr}_d(\aFq^n)$ given by fixing $L_W$ are called thin Schubert cells (cf.\ \cite[\S 2.1]{CV18}) .
For $W, W'\in\mathrm{Gr}_d(\aFq^n)$, if $W'$ is in the closure of the thin Schubert cell of $W$, then $L_{W'}\subseteq L_W$.
In general, the converse of this statement (the closure relation) does not hold for this stratification.
However, it holds if $d=2$ (see \cite[\S 1.10]{GS87}).

\begin{prop}
\label{EO 3 odd}
Assume that $n=7,8$.
Then $\SAdm(\mu)_0$ is equal to 
\begin{align*}
&\{\tau^3,s_0s_6\tau^3,s_0s_6s_1s_0\tau^3,s_0s_6s_5s_1\tau^3,s_0s_6s_5s_1s_0s_6\tau^3\} &(n=7),\\
&\{\tau^3,s_0s_1\tau^3,s_0s_7s_6s_5\tau^3,s_0s_7s_6s_1\tau^3,s_0s_7s_6s_5s_1s_0\tau^3,\\
&\hspace{3.94cm}s_0s_7s_6s_1s_0s_7\tau^3, s_0s_7s_6s_5s_1s_0s_7s_6\tau^3\} &(n=8).
\end{align*}
Each $w\in\SAdm(\mu)_0$ has positive Coxeter part, and $\J$ acts transitively on $\Irr X_w(b)$.
Each irreducible component of $\pi(X_w(b))(\cong X_w(b))$ is a $\J$-stratum universally homeomorphic to an affine space of dimension $\frac{\ell(w)}{2}$.
Moreover, the closure relation can be described in terms of $\cB(\J,F)$.
\end{prop}
\begin{proof}
Except the ``moreover'' part, the proposition follows from \cite[Proposition 6.5]{Shimada4} and \cite[Theorem 5.3]{Viehmann08}.
Let $X_\mu(b)^{i}=\{gK\in X_\mu(b)\mid \kappa(g)=i\}$.

Assume that $n=7$.
Then it follows from \cite[Lemma 5.2]{Shimada4} that $X_\mu(b)^4\subset K\vp^{\omega_4}K/K$.
Moreover, each $\J$-stratum in $X_\mu(b)^4$ of dimension $\le 2$ coincides with a Schubert cell in the Grassmannian $K\vp^{\omega_4}K/K$.
So the closure relation follows from this and the equidimensionality of $X_\mu(b)$.

Assume that $n=8$.
The $\J$-strata of dimension $\le 2$ in $X_\mu(b)^2$ coincide with some Schubert cells in the Grassmannian $K\vp^{\omega_2}K/K$.
So the closure relation holds for these strata.
By the equidimensionality of $X_\mu(b)$, it also holds for the $\J$-strata of dimension $4$.
It remains to show the closure relation for the $\J$-strata of dimension $3$.

We first treat the case of $s_0s_7s_6s_1s_0s_7\tau^3$.
Set $\ld=(0,1,0,0,1,0,0,0)$.
It follows from \cite[Lemma 5.2]{Shimada4} that $\tau^2\pi(X_{s_0s_7s_6s_1s_0s_7\tau^3}(b)^0)\subset \tau^2I\vp^{\chi^{\vee}_{3,7}}K/K=I\vp^{\ld}K/K\subset K\vp^{\omega_2}K/K$.
Fix an isomorphism $\A^4\cong I\vp^{\ld}K/K$ which maps $(s,t,u,v)$ to $h_{s,t,u,v}\vp^{\ld}K$, where 
$$h_{s,t,u,v}\coloneqq\begin{pmatrix}
1 & 0 & 0 & 0 & 0 & 0 & 0 & 0 \\
s & 1 & 0 & 0 & 0 & 0 & 0 & 0 \\
0 & 0 & 1 & 0 & 0 & 0 & 0 & 0 \\
0 & 0 & 0 & 1 & 0 & 0 & 0 & 0 \\
t & 0 & u & v & 1 & 0 & 0 & 0 \\
0 & 0 & 0 & 0 & 0 & 1 & 0 & 0 \\
0 & 0 & 0 & 0 & 0 & 0 & 1 & 0 \\
0 & 0 & 0 & 0 & 0 & 0 & 0 & 1
\end{pmatrix}.$$
Then we have $$\vp^{-\ld}h_{s,t,u,v}^{-1}b\sigma(h_{s,t,u,v}\vp^{\ld})=\begin{pmatrix}
0 & 0 & 0 & 0 & 0 & \vp & 0 & 0 \\
0 & 0 & 0 & 0 & 0 & -s & 1 & 0 \\
0 & 0 & 0 & 0 & 0 & 0 & 0 & \vp \\
1 & 0 & 0 & 0 & 0 & 0 & 0 & 0 \\
\frac{\sigma(s)-v}{\vp} & 1 & 0 & 0 & 0 & -t & 0 & -u \\
0 & 0 & 1 & 0 & 0 & 0 & 0 & 0 \\
0 & 0 & 0 & 1 & 0 & 0 & 0 & 0 \\
\sigma(t) & 0 & \sigma(u) & \sigma(v) & \vp & 0 & 0 & 0
\end{pmatrix}.$$
It is easy to check that $\tau^2\pi(X_{s_0s_7s_6s_1s_0s_7\tau^3}(b)^0)\cong \A^3$ is the locus $v=\sigma(s)$ in $I\vp^{\ld}K/K\cong \A^4$.
The locus $s=t=v=0$ and $u\neq 0$ in $\tau^2\pi(X_{s_0s_7s_6s_1s_0s_7\tau^3}(b)^0)$ is a thin Schubert cell whose list is $\{\{1,3,4,6,7,8\},\{1,4,5,6,7,8\}\}$.
Since the closure relation holds for thin Schubert cells in $K\vp^{\omega_2}K/K$, the closure of this Schubert cell intersect $\tau^2I\vp^{\chi_{1,7}^\vee}K/K$.
Indeed, the latter contains the thin Schubert cell with the list  $\{\{1,4,5,6,7,8\}\}$.
The locus $v=\sigma(s)\neq 0$ and $t=u=0$ in $\tau^2\pi(X_{s_0s_7s_6s_1s_0s_7\tau^3}(b)^0)$ is contained in a thin Schubert cell whose list is $$L\coloneqq\{\{1,3,4,6,7,8\},\{1,3,5,6,7,8\},\{2,3,4,6,7,8\},\{2,3,5,6,7,8\}\}.$$
Let $L'$ be the list of a thin Schubert cell in $\tau^2I\vp^{\chi_{1,7}^\vee}K/K,\tau^2I\vp^{\chi_{1,8}^\vee}K/K$ and $\{\tau^2K\}$.
Then it is easy to check that $L'\nsubseteq L$.
On the other hand, the closure of the locus $v=\sigma(s)\neq 0$ and $t=u=0$ in  $\tau^2\pi(X_{s_0s_7s_6s_1s_0s_7\tau^3}(b)^0)$ is projective.
By \cite[Proposition 2.11 (5)]{CV18}, this closure is contained in $$\tau^2I\vp^{\chi^{\vee}_{3,7}}K/K\sqcup\tau^2I\vp^{\chi_{1,7}^\vee}K/K\sqcup \tau^2I\vp^{\chi_{2,8}^\vee}K/K\sqcup\tau^2I\vp^{\chi_{1,8}^\vee}K/K\sqcup\{\tau^2K\}.$$
So it must intersect $\tau^2I\vp^{\chi_{2,8}^\vee}K/K$.
Thus both $\tau^2I\vp^{\chi_{1,7}^\vee}K/K$ and $\tau^2I\vp^{\chi_{2,8}^\vee}K/K$ intersect the closure of $\tau^2\pi(X_{s_0s_7s_6s_1s_0s_7\tau^3}(b)^0)$.
This combined with Lemma \ref{codim=1} imply that they are actually contained in the closure.
Note that $s_0s_7s_6s_1\tau^3\le s_0s_7s_6s_1s_0s_7\tau^3$ and $$s_7s_2(s_0s_7s_6s_5\tau^3)s_2s_7=s_0s_7s_0s_6\tau^3\le s_0s_7s_6s_1s_0s_7\tau^3.$$
Thus the closure relation holds.

We next treat the case of $s_0s_7s_6s_5s_1s_0\tau^3$.
Set $\ld=(1,1,1,0,1,1,0,1)$.
It follows from \cite[Lemma 5.2]{Shimada4} that $\tau^6\pi(X_{s_0s_7s_6s_5s_1s_0\tau^3}(b)^0)\subset \tau^6I\vp^{\chi^{\vee}_{2,6}}K/K=I\vp^{\ld}K/K\subset K\vp^{\omega_6}K/K$.
Fix an isomorphism $\A^4\cong I\vp^{\ld}K/K$ which maps $(s,t,u,v)$ to $h_{s,t,u,v}\vp^{\ld}K$, where 
$$h_{s,t,u,v}\coloneqq\begin{pmatrix}
1 & 0 & 0 & 0 & 0 & 0 & 0 & 0 \\
0 & 1 & 0 & 0 & 0 & 0 & 0 & 0 \\
0 & 0 & 1 & 0 & 0 & 0 & 0 & 0 \\
0 & 0 & 0 & 1 & 0 & 0 & 0 & 0 \\
0 & 0 & 0 & s & 1 & 0 & 0 & 0 \\
0 & 0 & 0 & t & 0 & 1 & 0 & 0 \\
0 & 0 & 0 & 0 & 0 & 0 & 1 & 0 \\
0 & 0 & 0 & u & 0 & 0 & v & 1
\end{pmatrix}.$$
Then we have $$\vp^{-\ld}h_{s,t,u,v}^{-1}b\sigma(h_{s,t,u,v}\vp^{\ld})=\begin{pmatrix}
0 & 0 & 0 & \sigma(t) & 0 & \vp & 0 & 0 \\
0 & 0 & 0 & 0 & 0 & 0 & 1 & 0 \\
0 & 0 & 0 & \sigma(u) & 0 & 0 & \sigma(v) & \vp \\
\vp & 0 & 0 & 0 & 0 & 0 & 0 & 0 \\
-s & 1 & 0 & 0 & 0 & 0 & 0 & 0 \\
-t & 0 & 1 & 0 & 0 & 0 & 0 & 0 \\
0 & 0 & 0 & 1 & 0 & 0 & 0 & 0 \\
-u & 0 & 0 & \frac{\sigma(s)-v}{\vp} & 1 & 0 & 0 & 0
\end{pmatrix}.$$
It is easy to check that $\tau^6\pi(X_{s_0s_7s_6s_5s_1s_0\tau^3}(b)^0)\cong \A^3$ is the locus $v=\sigma(s)$ in $I\vp^{\ld}K/K\cong \A^4$.
The locus $s=u=v=0$ and $t\neq 0$ in $\tau^6\pi(X_{s_0s_7s_6s_5s_1s_0\tau^3}(b)^0)$ is a thin Schubert cell whose list is $\{\{4,7\},\{6,7\}\}$.
Since the closure relation holds for thin Schubert cells in $K\vp^{\omega_6}K/K$, the closure of this Schubert cell intersect $\tau^6I\vp^{\chi_{2,8}^\vee}K/K$.
Indeed, the latter contains the thin Schubert cell with the list  $\{\{6,7\}\}$.
The locus $v=\sigma(s)\neq 0$ and $t=u=0$ in $\tau^6\pi(X_{s_0s_7s_6s_5s_1s_0\tau^3}(b)^0)$ is contained in a thin Schubert cell whose list is $$L\coloneqq\{\{4,7\},\{4,8\},\{5,7\},\{5,8\}\}.$$
Let $L'$ be the list of a thin Schubert cell in $\tau^6I\vp^{\chi_{1,7}^\vee}K/K,\tau^6I\vp^{\chi_{1,8}^\vee}K/K$ and $\{\tau^6K\}$.
Then it is easy to check that $L'\nsubseteq L$.
On the other hand, the closure of the locus $v=\sigma(s)\neq 0$ and $t=u=0$ in  $\tau^6\pi(X_{s_0s_7s_6s_5s_1s_0\tau^3}(b)^0)$ is projective.
By \cite[Proposition 2.11 (5)]{CV18}, this closure is contained in $$\tau^6I\vp^{\chi^{\vee}_{2,6}}K/K\sqcup\tau^6I\vp^{\chi_{1,7}^\vee}K/K\sqcup \tau^6I\vp^{\chi_{2,8}^\vee}K/K\sqcup\tau^6I\vp^{\chi_{1,8}^\vee}K/K\sqcup\{\tau^6K\}.$$
So it must intersect $\tau^6I\vp^{\chi_{2,8}^\vee}K/K$.
Thus both $\tau^6I\vp^{\chi_{1,7}^\vee}K/K$ and $\tau^6I\vp^{\chi_{2,8}^\vee}K/K$ intersect the closure of $\tau^6\pi(X_{s_0s_7s_6s_5s_1s_0\tau^3}(b)^0)$.
This combined with Lemma \ref{codim=1} imply that they are actually contained in the closure.
Note that $s_0s_7s_6s_1\tau^3\le s_0s_7s_6s_5s_1s_0\tau^3$ and $s_0s_7s_6s_5\tau^3\le s_0s_7s_6s_5s_1s_0\tau^3$.
Thus the closure relation holds.
This finishes the proof.
\end{proof}

\subsection{The cases of $\omega_1+\omega_2$ and $3\omega_1$ when $n=3,4,5$}

\begin{prop}
\label{EO 12}
Assume that $n=4,5$.
Then $\SAdm(\omega_1+\omega_2)_0$ is equal to 
\begin{align*}
&\{s_0s_3s_2s_1\tau^3,s_0s_1s_3s_0\tau^3,s_0s_3\tau^3\}\sqcup\SAdm(\omega_3)_0 &(n=4),\\
&\{s_0s_4s_3s_2s_1s_0\tau^3,s_0s_1s_4s_3s_0s_4\tau^3,s_0s_4s_3s_2\tau^3,s_0s_1s_4s_3\tau^3\}\sqcup\SAdm(\omega_3)_0 &(n=5).
\end{align*}
Also $\SAdm(3\omega_1)_0$ is equal to 
\begin{align*}
&\{s_0s_3s_2s_1s_0s_3\tau^3\}\sqcup\SAdm(\omega_1+\omega_2)_0 &(n=4),\\
&\{s_0s_4s_3s_2s_1s_0s_4s_3\tau^3\}\sqcup\SAdm(\omega_1+\omega_2)_0 &(n=5).
\end{align*}
Each $w\in\SAdm(3\omega_1)_0$ has positive Coxeter part, and $\J$ acts transitively on $\Irr X_w(\tau^3)$.
Each irreducible component of $\pi(X_w(\tau^3))(\cong X_w(\tau^3))$ is a $\J$-stratum universally homeomorphic to an affine space of dimension $\frac{\ell(w)}{2}$.
Moreover, the closure relation can be described in terms of $\cB(\J,F)$.
\end{prop}
\begin{proof}
Except the ``moreover'' part, the proposition follows from \cite[Proposition 6.5 \& Theorem 7.2]{Shimada4}.
The ``moreover'' part follows from the equidimensionality of $X_{\pc \mu}(b)$ and the theory of thin Schubert cells similarly as the proof of Proposition \ref{EO 3 odd}.
\end{proof}

\section{The cases of $3\omega_1$, $2\omega_1+\omega_2$, $4\omega_1$ and $\omega_1+3\omega_2$}
Keep the notation above.
We assume that $n=3$.

\subsection{The case of $3\omega_1$ when $n=3$}
In this subsection, we set $\mu=3\omega_1$.
For $i,i'\in \{0,1,2\}$, let $P_{ii'}$ denote the standard parahoric subgroup of $G(L)$ corresponding to $\{s_i,s_{i'}\}$.
Let $\Omega^2_{\F_{q}}$ denote the (perfection of) $2$-dimensional Drinfeld's upper half-space over $\F_{q}$.
Then the Deligne-Lusztig variety $\{g\in G(\aFq)/B(\aFq)\mid \inv(g, \sigma(g))=s_1s_2\}$ of Coxeter type is isomorphic to $\Omega^2_{\F_{q}}$.
\begin{prop}
We have
$$\SAdm(\mu)_0=\{\tau^3,s_0\tau^3,s_0s_1\tau^3,s_0s_2\tau^3,s_0s_2s_1s_0\tau^3\}.$$
Every element of $\SAdm(\mu)_0$ has positive Coxeter part.
Set \begin{align*}
P_{\tau^2}=P_{12},\quad P_{s_0\tau^3}=P_{01}\cap P_{02},\quad P_{s_0s_1\tau^3}=P_{01},\quad P_{s_0s_2\tau^3}=P_{02},\quad P_{s_0s_2s_1s_0\tau^3}=P_{12}.
\end{align*}
For any $w\in \SAdm(\mu)_0$, there exists an irreducible component $Y(w)$ of $X_w(b)$ such that $X_w(b)=\J Y(w)$, $Y(w)\cong \pi(Y(w))$ and $$\pi(X_{w}(b))=\bigsqcup_{j\in \J/\J\cap P_{w}}j \pi(Y(w)).$$
Each $j\pi(Y(w))$ is a $\J$-stratum of $X_{\pc\mu}(b)$ with
\begin{align*}
&\pi(Y(\tau^2))\cong \{pt\},\quad \pi(Y(s_0\tau^3))\cong \mathbb P^1\setminus \mathbb P^1(\F_{q}),\quad \pi(Y(s_0s_1\tau^3))\cong \Omega^2_{\F_{q}},\\
&\pi(Y(s_0s_2\tau^3))\cong \Omega^2_{\F_{q}},\quad
\pi(Y(s_0s_2s_1s_0\tau^3))\cong \Omega^2_{\F_{q}}\times \A^1.
\end{align*}
Moreover, the closure relation can be described in terms of $\cB(\J,F)$.
\end{prop}
\begin{proof}
The proof is similar to Proposition \ref{EO 3 even}.
See also \cite{Shimada}.
\end{proof}

\subsection{The cases of $2\omega_1+\omega_2$, $4\omega_1$ and $\omega_1+3\omega_2$ when $n=3$}
The following proposition follows similarly as Proposition \ref{EO 3 odd} (see also \cite{Shimada}).
\begin{prop}
Assume that $n=3$.
Then
$$\SAdm(2\omega_1+\omega_2)_0=\{s_0s_1s_2s_1\tau^4,s_0s_2s_1s_0\tau^4\}\sqcup \SAdm(2\omega_2)_0.$$
Also we have
\begin{align*}
\SAdm(4\omega_1)_0&=\{s_0s_2s_1s_0s_2s_1\tau^4\}\sqcup \SAdm(2\omega_1+\omega_2)_0,\\
\SAdm(\omega_1+3\omega_2-\omega_3)_0&=\{s_0s_1s_2s_0s_1s_2\tau^4,s_0s_1s_2s_1s_0s_1\tau^4\}\sqcup \SAdm(2\omega_1+\omega_2)_0.
\end{align*}
Each element $w$ of these sets has positive Coxeter part, and $\J$ acts transitively on $\Irr X_w(\tau^4)$.
Each irreducible component of $\pi(X_w(\tau^4))(\cong X_w(\tau^4))$ is a $\J$-stratum universally homeomorphic to an affine space of dimension $\frac{\ell(w)}{2}$.
Moreover, the closure relation can be described in terms of $\cB(\J,F)$.
\end{prop}

\section{The cases of $m\omega_1\ \text{with $m\in \Z_{>0}$}$}
Keep the notation above.
In this section, we set $\mu=m\omega_1$ and $b=\tau^m$ with $m\in \Z_{>0}$.
We assume that $n=2$.
Then the $F$-rank of $\J=\J_b$ is $0$ (resp.\ $1$) if $m$ is odd (resp.\ even).

\subsection{The cases of $m\omega_1\ \text{with $m\in \Z_{>0}$}$ when $n=2$ and $m$ is odd}
Assume that $m$ is odd.
For any $1\le k\le \frac{m-1}{2}(=\dim X_{\mu}(b))$, set
$$\ld_k=
\begin{cases}
\frac{k+1}{2}\chi_{1,2}^\vee & (k\ \text{odd})\\
\frac{k}{2}\chi_{2,1}^\vee & (k\ \text{even}).
\end{cases}$$
We also set $\ld_0=\omega_0$.
The following proposition follows similarly as Proposition \ref{EO 2 odd} (see also \cite{Ivanov13}).
\begin{prop}
We have
$$\SAdm(\mu)=\{\tau^m, s_0s_1\tau^m, s_0s_1s_0s_1\tau^m,\ldots, s_0s_1\cdots s_0s_1\tau^m\},$$
where the last element has length $m-1$.
Every element of $\SAdm(\mu)_0$ has positive Coxeter part.
For $0\le k\le \frac{m-1}{2}$, let $w_k$ denote the unique element in $\SAdm(\mu)_0$ of length $2k$.
Then there exists an irreducible component $Y(w_k)$ of $X_{w_k}(b)$ such that $X_{w_k}(b)=\J Y(w_k)$, $Y(w_k)\cong\pi(Y(w_k))=X_{\mu}^{\ld_k}(b)\cong \A^k$ and
$$\pi(X_{w_k}(b))=\bigsqcup_{j\in \J/\J\cap I}j X_{\mu}^{\ld_k}(b).$$
Moreover, the closure relation can be described in terms of $\cB(\J,F)$.
\end{prop}

\subsection{The cases of $m\omega_1\ \text{with $m\in \Z_{>0}$}$ when $n=2$ and $m$ is even}
Assume that $m$ is even.
For $i\in \{0,1\}$, let $P_i$ denote the standard parahoric subgroup of $G(L)$ corresponding to $\{s_i\}$.
The following proposition follows similarly as the results in \S \ref{EO 2 even} (see also \cite{Ivanov13}).
\begin{prop}
We have
$$\SAdm(\mu)_0=\{\tau^m, s_0\tau^m,s_0s_1s_0\tau^m,\ldots, s_0s_1s_0\cdots s_1s_0\tau^m\},$$
where the last element has length $m-1$.
Every element of $\SAdm(\mu)_0$ has positive Coxeter part.
For $1\le k\le \frac{m}{2}$, let $w_k$ denote the unique element in $\SAdm(\mu)_0$ of length $2k-1$.
Set $w_0=\tau^m$.
Then there exists an irreducible component $Y(w_k)$ of $X_{w_k}(b)$ such that $X_{w_k}(b)=\J Y(w_k)$, $Y(w_k)\cong\pi(Y(w_k))\cong (\mathbb P^1\setminus \mathbb P^1(\Fq))\times \A^{k-1}$ and
$$\pi(X_{w_k}(b))=\bigsqcup_{j\in \J/\J\cap P_{w_k}}j \pi(Y(w_k)).$$
Here $P_{w_k}=P_0$ (resp.\ $P_1$) if $k$ is odd (resp.\ even).
Moreover, the closure relation can be described in terms of $\cB(\J,F)$.
\end{prop}

\bibliographystyle{myamsplain}
\bibliography{reference}
\end{document}